\documentclass[reqno, 11pt]{amsart}

\usepackage[english]{babel}
\usepackage{amsmath}
\usepackage{amsthm}
\usepackage{amsfonts}
\usepackage{amssymb}
\usepackage{anysize}
\usepackage{hyperref}
\usepackage{appendix}
\usepackage{mathtools}
\usepackage{graphicx}
\usepackage{xcolor}
\usepackage{amsmath,esint,bigints, amssymb, mathrsfs, amsthm,mathbbol,upgreek,exscale}

\newtheorem{defi}{Definition}[section]
\newtheorem{lem}[defi]{Lemma}
\newtheorem{theo}[defi]{Theorem}

\newtheorem{pro}[defi]{Proposition}

\newtheorem{rem}[defi]{Remark}

\DeclareMathOperator{\N}{\mathbb{N}}
\DeclareMathOperator{\R}{\mathbb{R}}
\DeclareMathOperator{\C}{\mathbb{C}}
\DeclareMathOperator{\T}{\mathbb{T}}
\DeclareMathOperator{\D}{\mathbb{D}}
\DeclareMathOperator{\Z}{\mathbb{Z}}

\allowdisplaybreaks 

\title[]{Steady solutions for the Schr\"odinger map equation}

\author[C. Garc\'ia]{Claudia Garc\'ia}
\address{ Departamento de Matem\'aticas, Universidad Aut\'onoma de Madrid, Ciudad Universitaria de Cantoblanco, 28049, Madrid, Spain \& Research Unit ``Modeling Nature'' (MNat), Universidad de Granada, 18071 Granada, Spain}
\email{ claudia.garcial@uam.es}

\author{Luis Vega}
\address{Departamento de Matem\'aticas, Universidad del Pais Vasco, Apdo 644, 48080, Bilbao, Spain \& BCAM, Mazarredo 14, 48009, Bilbao, Spain}
\email{lvega@bcamath.org}

\thanks{ C.G. has been supported by the ERC-StG-852741 (CAPA), the MINECO--Feder (Spain) research grant number RTI2018--098850--B--I00, the Junta de Andaluc\'ia (Spain) Project
FQM 954, the Severo Ochoa Programme for Centres of Excellence in R\&D(CEX2019-000904-S) and by the PID2021-124195NB-C32, the  ERC Advanced Grant 834728 and by the CAM under the multiannual Agreement with UAM in the line for the Excellent of the University Research Staff in the context of the V PRICIT,   and  L. V. has been supported by MINECO grant PGC2018-094522-B-I00 (Spain) and IT1247-19 (Gobierno Vasco). This work was developed during the semester program {\it Hamiltonian Methods in Dispersive and Wave Evolution Equations} at ICERM.}

\begin{document}

\date{\today}

\begin{abstract}
In this paper we use bifurcation methods to construct a new family of solutions of the binormal flow, also known as the vortex filament equation, which do not change their form. Our examples are complementary to those obtained by S. Kida in 1981, and therefore they are also related, thanks to the so-called Hasimoto transformation, to travelling wave solutions of the 1d cubic non-linear Schrödinger equation.
\end{abstract}

\maketitle

\tableofcontents

\section{Introduction}

We are interested in space periodic solutions of the vortex filament equation
\begin{equation}\label{intro-equation-1}
\mathbf{X}_t=\mathbf{X}_s\wedge_+ \mathbf{X}_{ss},
\end{equation}
where $\wedge_+$ is the usual cross-product, and $s$ is the arclength. This equation was proposed by Da Rios in \cite{DaR} as a simplified model that describes the evolution of vortex filaments, see also \cite{ArHa} and \cite{Be}. Some simple examples of solutions are the circle and the straight filament which are easily obtained because the curvature is constant and have zero torsion. Looking for traveling solutions for the tangent vector $\bf{T}=\bf{X}_s$ one  obtains helices which are characterised for having constant curvature and constant torsion. Moreover,  Kida \cite{Kida81, Kida82} found a rich family of steady solutions that can be described in terms of elliptic integrals of the first, second, and third kind. Indeed, the evolution of this family consists in a rotation with constant angular velocity, a translation with constant speed, and a slipping motion.

The aim of this paper is to study space periodic solutions to \eqref{equation-1} close to the circle or the helix of any radius and any pitch, that are different from the Kida's solutions. 
Motivated by the amount of work done in the last decades about periodic vortex patches, see \cite{HM, HMV, Hmidi-Renault} and references therein, we will implement the bifurcation theory in terms of the well-known Crandall-Rabinowitz theorem (see Section \ref{sec-CR}) and find steady solutions which are perturbations of the helices. It is relevant to observe that helices are included in Kida's family and they are studied in Section 4.2 of \cite{Kida81}.  In this paper we propose an alternative procedure to look for perturbations of helices to the one proposed by Kida in Section 4.3 of the same article. Our construction will also give for general helices the instability result obtained by Jerrard and Smets in \cite{Jerrard-Smets} for circles and straight lines.

For doing this we will not directly work with \eqref{intro-equation-1} but with the tangent vector $\bf{T}=\bf{X}_s$. It is immediate to obtain from \eqref{intro-equation-1} that the equation that $\bf{T}$ has to satisfy is
\begin{equation}\label{intro-equation-2}
\mathbf{T}_t=\mathbf{T}\wedge_+ \mathbf{T}_{ss}.
\end{equation}
Here, both the circle and the helix can be seen as circles in the sphere. We will project these solutions to the plane with the help of the stereographic projection,
\begin{equation}\label{intro-z-equation}
z=x+iy\equiv (x,y)\equiv\left(\frac{T_1}{1+T_3},\frac{T_2}{1+T_3}\right),
\end{equation}
which satisfies 
\begin{equation}\label{intro-z-equation-2}
z_t=i z_{ss}- \frac{2i\overline{z}}{1+|z|^2}z_s^2,
\end{equation}
see for example \cite{Hoz-GarciaCervera-Vega, Vega2015} for more details about the deduction of this equation. Here, we will be interested in  solutions of the form
\begin{equation}\label{intro-evolution-trivial}
z(t,s)=e^{i\Omega t}z_0(s-at),
\end{equation}
that is, we assume some rotating and slipping motion for $z$. Then, \eqref{intro-z-equation-2} agrees with
\begin{equation}\label{intro-z-equation-3}
\Omega z_0(\xi)+aiz_0'(\xi)-z_0''(\xi)+ \frac{2\overline{z_0(\xi)}z_0'(\xi)^2}{1+ |z_0(\xi)|^2}=0,
\end{equation}
where $\xi=s-at$. We define $z_0(\xi)=g(e^{i\xi})=g(w)$, $R\geq 0$, and
$$
g(w)=w(R+f(w)), \quad \Omega=\Omega_R+\lambda, \quad \Omega_R=a+\frac{-1+R^2}{1+R^2},
$$
where $f$ is a perturbation of the helix. Then, we have the equivalent equation
\begin{equation}\label{equation-intro}
G_a(R,\lambda, f)(w)=0, \quad \forall w\in\T,
\end{equation}
with
\begin{align}\label{intro-G-eq}
G_a(R,\lambda,f)(w):=&\Omega R+\Omega f(w)+(1-a)R+(1-a)f(w)+(3-a)wf'(w)+w^2f''(w)\\
&- 2\frac{(R+\overline{f(w)})(R+f(w)+wf'(w))^2}{1+ (R^2+Rf(w)+R\overline{f(w)}+f(w)\overline{f(w)})}.\nonumber
\end{align}
Hence, finding solutions to \eqref{intro-equation-1} with the ansatz \eqref{intro-evolution-trivial} agrees with looking for the roots of $G_a$, where a trivial curve of solutions is known:
\begin{equation}\label{trivial-sol}
G_a(R,0,0)=0, \quad \forall R\geq 0,
\end{equation}
which correspond to the steady helices. The circle refers here to $a=0$ and $R=1$ implying that $\Omega_1=0$.

The resolution of a nonlinear equation of the type \eqref{equation-intro} together with the existence of a trivial line of solutions \eqref{trivial-sol} can be studied through bifurcation theory. In particular, here we will apply the well-known Crandall-Rabinowitz theorem depending on parameters (in this case: $a$). Hence, the existence of nontrivial roots of \eqref{equation-intro} reduces to the spectral study of the linearized operator at the trivial solution: $\partial_f G_a(R,0,0)$. In particular, one needs to show that
\begin{equation}\label{intro-kernel}
\dim \textnormal{Ker} \partial_f G_a(R^\star,0,0)\neq 0,
\end{equation}
for some $R^\star$, which is  often called as an {\it eigenvalue}. After adding the extra parameter $m$ defining the $m$-fold symmetry of the solution, we find the following eigenvalues. For $a=0$, hence only $R^\star=1$ is found for $m=1$. In such case, eigenvalues do not exist for other symmetries $m>1$. Letting the parameter $a$ to be free, we find two family of eigenvalues:
$$
R^\star=R_m^{e,-}:=\sqrt{\frac{- \left(m^2+2-a^2\right)- 2\sqrt{3m^2-3+a^2}}{m^2-4-4a-a^2}},
$$
for $\sqrt{a^2-2}<m<a+2$ if $a>\sqrt{2}$, and $m<a+2$ for $a\leq \sqrt{2}$.  Moreover, in the case $a<1$, the symmetry $m=1$ is also allowed associated to the eigenvalue:
$$
R^\star=R_1^{e,+}:=\sqrt{\frac{-3+a^2+ 2a}{-3-4a-a^2}}.
$$
Hence, fixing $a$, there exists a finite number of eigenvalues parametrized by the symmetry $m$.

After finding such eigenvalues, the transversal condition of Crandall-Rabinowitz theorem has to be satisfied, and this reads as
\begin{align}\label{transversal}
\left(\frac{(m+a)(1+ (R^\star)^2)-2(1- (R^\star)^2)}{(m-a)(1+ (R^\star)^2)+2(1- (R^\star)^2)}\right)^2
&\neq \frac{ 1- (R^\star)^3-2m(1+ (R^\star)^2) }{ 1- (R^\star)^3+2m(1+ (R^\star)^2)   }.
\end{align}

From such eigenvalues satisfying \eqref{transversal}, we can apply Crandall-Rabinowitz theorem to find a local curve of solutions which is transversal to the trivial one. Our main result, in a simplified version, is as follows.

\begin{theo}\label{intro-theorem}
\begin{itemize}
\item For $a<1$, there exists a nontrivial curve of solutions $(R(\eta), \lambda(\eta), f(\eta))$, with $\eta\in I$, to \eqref{equation-intro} where $f$ is $1$-fold symmetric, which bifurcates from
$$
R_1^{e,+}=\sqrt{\frac{-3+a^2+ 2a}{-3-4a-a^2}},
$$
only if \eqref{transversal} is satisfied.
\item For  $a$ fiixed, there exists a nontrivial curve of solutions $(R(\eta), \lambda(\eta), f(\eta))$, with $\eta\in I$, to \eqref{equation-intro} which bifurcates from
$$
R_m^{e,-}=\sqrt{\frac{ \left(m^2+2-a^2\right)+2\sqrt{3m^2-3+a^2}}{4+4a+a^2-m^2}},
$$
with $m$-fold symmetry, where $m$ verifies: $\sqrt{a^2-2}<m<a+2$ if $a>\sqrt{2}$, and $m<a+2$ for $a\leq \sqrt{2}$; only if \eqref{transversal} is satisfied.
\end{itemize}
\end{theo}
Let us mention that \eqref{transversal} is checked for the case $a=0$ in Proposition \ref{prop-transversal}, and can be also checked for other $a\neq 0$. Moreover, let us recall that we are finding solutions $f$, but we can come back to the original variable $z$:
\begin{equation}\label{intro-nontrivial-z}
z_0(s)=g(e^{is})=g(w)=w(R+f(w)), \quad z(t,s)= e^{i\Omega t}z_0(s-at).
\end{equation}

An important subclass of the family constructed in Theorem 1.1 is obtained taking $m=a+\theta$ with $0\leq\theta\leq 2$. For that choice we have that if
$$\lim_{m\to \infty} R_m^{e,-}:= R_\theta=\sqrt{1+\frac{2\theta}{2-\theta}},
$$
then $1\leq R_\theta\leq\infty$. As a consequence, we can approach any given helix by a sequence of solutions with analogous properties than those used in \cite[Theorem 2]{Jerrard-Smets}  for the circle.

Once we have obtained non trivial solutions \eqref{intro-nontrivial-z} we can come back to ${\bf T^0}$:
$$
\mathbf{T}^0=\frac{1}{1\pm |z_0(s)|^2}\left(2z_0(s),1\mp |z_0(s)|^2\right),\quad \mathbf{T}(t,s)=\mathscr{R}^{\Omega t}\mathbf{T}^0(s-at),
$$
where $\mathscr{R}^{\Omega t}{\bf x}=(e^{i\Omega t}(x_1,x_2),x_3).$ Hence, we find solutions for the Schr\"odinger map equation which are rotations of the initial one. Moreover, the initial tangent vector is an appropriate perturbation of the helix one.

As for ${\bf X}$ we get
\begin{align*}
\mathbf{X}=&\mathscr{R}^{\Omega t}\int_{0}^s \mathbf{T}^0(\sigma-at)d\sigma+\int_0^t \mathscr{R}^{\Omega\tau}\left(\mathbf{T^0}(-a\tau)\wedge_\pm \mathbf{T}_s^0(-a\tau)\right))d\tau.
\end{align*}

\begin{figure}[h]
\minipage{0.5\textwidth}
\begin{center}
\includegraphics[width=0.7\textwidth]{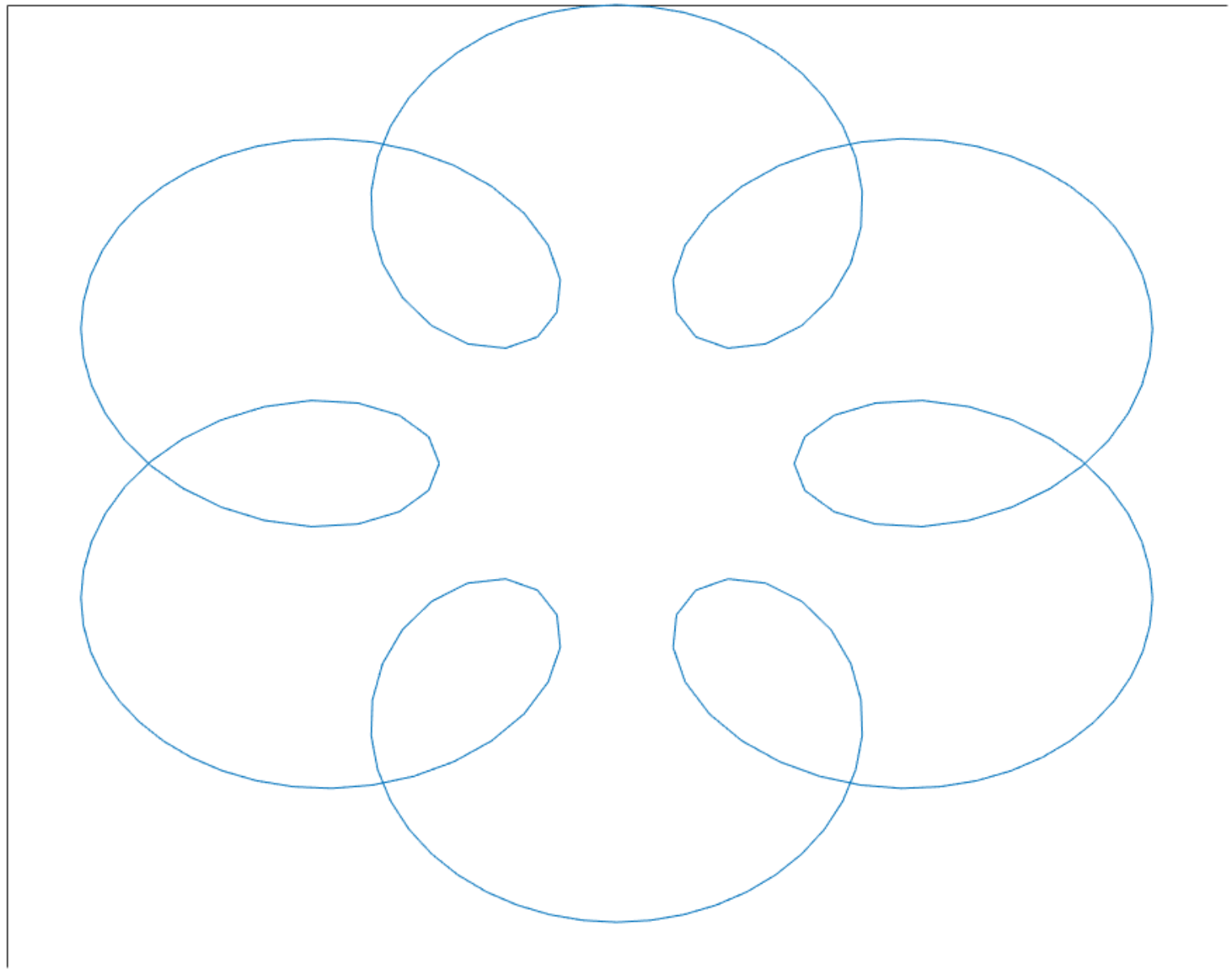}
\end{center}
\endminipage\hfill
\minipage{0.5\textwidth}
\begin{center}
\includegraphics[width=0.8\textwidth]{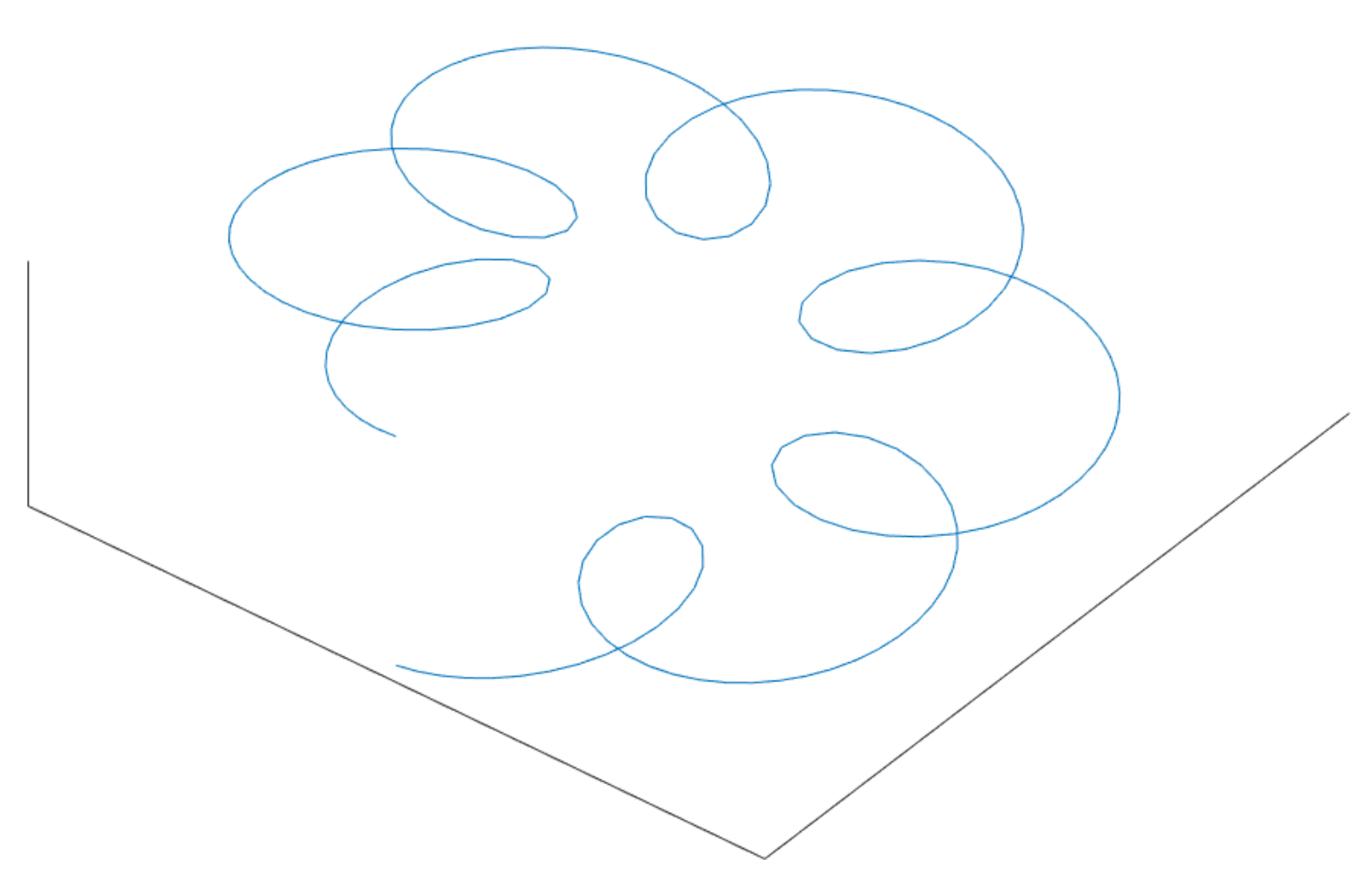}
\end{center}
\endminipage\hfill
%
\caption{$m=6$, $a=0$, initial time}
\end{figure}

Let us remark that we will also cover the case with a hyperbolic underlying  geometry. It is  given by the equation
\begin{equation}\label{intro-equation-1-hyp}
\mathbf{X}_t=\mathbf{X}_s\wedge_- \mathbf{X}_{ss},
\end{equation}
with $\wedge_-$ defined as
$$
\mathbf{a}\wedge_-\mathbf{b}=(a_2b_3-a_3b_2,a_3b_1-a_1b_3,-(a_1b_2-a_2b_1).
$$
Here, we can also use the stereographic projection to find the equation of $z$. After decomposing $z$ in terms of $f$, we can relate the existence of nontrivial rotating and slipping solutions to \eqref{intro-equation-1-hyp} with the roots of
$$
G_a^h(R,\lambda,0)=0,
$$
where
\begin{align}\label{intro-G-eq-2}
G_a^h(R,\lambda,f)(w):=&\Omega R+\Omega f(w)+(1-a)R+(1-a)f(w)+(3-a)wf'(w)+w^2f''(w)\\
&+ 2\frac{(R+\overline{f(w)})(R+f(w)+wf'(w))^2}{1- (R^2+Rf(w)+R\overline{f(w)}+f(w)\overline{f(w)})},
\end{align}
with $\Omega=\Omega_R+\lambda$ and 
$$
\Omega_R=a+\frac{-1-R^2}{1-R^2}.
$$
In the hyperbolic case, the family of eigenvalues, i.e., the values $R$ such that the kernel of the linear operator is not trivial, is richer. Indeed, we can find an infinite number of curves parametrized by $m>m_0$, for some $m_0$, where $m$ refers to the $m$-fold symmetry. The statement of the main result in the hyperbolic case is postpone to Section \ref{sec-main-result}.

The paper is organized as follows. In Section \ref{sec-equations}, we will review the equivalent equations to \eqref{intro-equation-1} in terms of the Schrodinger map. Later, we will find the equation for the rotating and slipping solutions and how we perturb the trivial ones (the circle and the helix). Later, we will introduce the function spaces that will be used in order to implement the Crandall-Rabinowitz theorem, which will be presented in Section \ref{sec-CR}. Then, Sections \ref{sec-spectral-0} and \ref{sec-spectral-a} aims to provide the spectral study for the cases $a=0$ (without slipping motion) and $a\neq 0$. Although the first case is covered also in Section \ref{sec-spectral-a}, we give all the details for the sake of clarity.  Section \ref{sec-main-result} gives the main result of this paper together with the proof. Finally, in Section \ref{sec-kida}, we compare our solutions with those ones obtained by Kida in \cite{Kida81}.

\section{Formulation}\label{sec-equations}

\subsection{Schrodinger map equation}

Given a curve $\textbf{X}_0:\R\rightarrow \R^3$, consider the geometric flow
$$
\mathbf{X}_t=c\mathbf{b},
$$
where $c$ is the curvature and ${\bf b}$ is the binormal component of the Frenet-Serret formulae
\begin{equation}\label{Frenet}
\left(\begin{array}{l}
\textbf{T}\\ \textbf{n} \\ \textbf{b}
\end{array}\right)_s=
\left(
\begin{array}{lll}
0&c&0\\
-c&0&\tau\\
0&-\tau&0
\end{array}
\right)
\cdot
\left(\begin{array}{l}
\textbf{T}\\ \textbf{n} \\ \textbf{b}
\end{array}\right).
\end{equation}
Then, the flow can be expressed as 
\begin{equation}\label{equation-1}
\mathbf{X}_t=\mathbf{X}_s\wedge_+ \mathbf{X}_{ss},
\end{equation}
where $\wedge_+$ is the usual cross-product, and $s$ is the arclength.


We can also derive an equation for the tangent vector $\mathbf{T}=\mathbf{X}_s$. First, note that it remains with constant length, and then we can assume that it takes values on the unit sphere. Differentiating \eqref{equation-1} we arrive at
\begin{equation}\label{equation-2}
\mathbf{T}_t=\mathbf{T}\wedge_+ \mathbf{T}_{ss}.
\end{equation}
This equation is known as the Schr\"odinger map equation on the sphere.

We can generalize \eqref{equation-2} by considering more complex varieties as its image. Indeed, we can choose also the hyperbolic plane $\mathbb{H}^2$ as the target space. In such a case, the equation for $\mathbf{T}$ reads as 
\begin{equation}\label{equation-2-hyp}
\mathbf{T}_t=\mathbf{T}\wedge_- \mathbf{T}_{ss},
\end{equation}
or equivalent for $\mathbf{X}$:
\begin{equation}\label{equation-1-hyp}
\mathbf{X}_t=\mathbf{X}_s\wedge_- \mathbf{X}_{ss},
\end{equation}
with $\wedge_-$ defined as
$$
\mathbf{a}\wedge_-\mathbf{b}=(a_2b_3-a_3b_2,a_3b_1-a_1b_3,-(a_1b_2-a_2b_1).
$$
In the same way, the equation for the curve can be generalized to
\begin{equation}\label{equation-1-pm}
\mathbf{X}_t=\mathbf{X}_s\wedge_\pm \mathbf{X}_{ss},
\end{equation}
with
$$
\mathbf{a}\wedge_\pm\mathbf{b}=(a_2b_3-a_3b_2,a_3b_1-a_1b_3,\pm(a_1b_2-a_2b_1).
$$
Equivalently, we define the modified the scalar product as
$$
\mathbf{a}\circ_\pm\mathbf{b}=a_1b_1+a_2b_2\pm a_3b_3.
$$
By using $\pm$ we can consider simultaneously the Euclidean case (corresponding to $+$) and the Hyperbolic case (corresponding to $-$). In what follows, when we use $\mp$ we will refer $-$ sign for the Euclidean case, and the $+$ sign for the Hyperbolic case. Then, the equation for the tangent vector agrees with 
\begin{equation}\label{equation-2-pm}
\mathbf{T}_t=\mathbf{T}\wedge_\pm \mathbf{T}_{ss}.
\end{equation}
If $\mathbf{T}\in \mathbb{H}^2$, we can give a generalized version of the Frenet-Serret trihedron \eqref{Frenet}, for each point of the curve $\mathbf{X}$, formed by $\mathbf{T}$ and two other vectors $\mathbf{e_1}$ and $\mathbf{e_2}$, see \cite{Hoz-GarciaCervera-Vega}.

In order to work better with $\mathbf{T}$, we will consider its stereographic projection over $\C$, see \cite{Hoz-GarciaCervera-Vega}. Define
\begin{equation}\label{z-equation}
z=x+iy\equiv (x,y)\equiv\left(\frac{T_1}{1+T_3},\frac{T_2}{1+T_3}\right).
\end{equation}
We are projecting $\mathbf{T}$ from $\mathbb{S}^2-(0,0,-1)$ into $\R^2$, identifying $\R^2$ with $\C$. In the Euclidean case, where $\mathbf{T}\in\mathbb{S}^2$, there is a point on the sphere, $(0,0,-1)$, to which no point in $\C$ corresponds, because the sphere is compact. Thus, we have a bijection between $\mathbb{S}^2-\{(0,0,-1)\}$ and $\R^2$. In the hyperbolic case, when $\mathbf{T}\in\mathbb{H}^2$, since $T_3>0$, we have a bijection between $\D$ and $\mathbb{H}^2$, where
$$
\D=\left\{(x,y)\in\R^2, \quad x^2+y^2<1\right\}.
$$
The tangent vector $\mathbf{T}$ can be recovered as
\begin{equation}\label{equation-T-z}
\mathbf{T}=(T_1,T_2,T_3)\equiv\left(\frac{2x}{1\pm x^2\pm y^2},\frac{2y}{1\pm x^2\pm y^2}, \frac{1\mp x^2\mp y^2}{1\pm x^2\pm y^2}\right).
\end{equation}
Differentiating in \eqref{z-equation} we arrive at the following nonlinear Schr\"odinger equation for $z$:
\begin{equation}\label{z-equation-2}
z_t=i z_{ss}\mp \frac{2i\overline{z}}{1\pm|z|^2}z_s^2.
\end{equation}

Explicit steady solutions for the equation of ${\bf X}$ are the circle and helix filaments. Those can be translated to ${\bf T}$ and are identified there as circles. We can go further and identify them in terms of $z$ as $z_0(s)=Re^{is}$. Such solutions can be observed to be rotating solutions of \eqref{z-equation-2}.

\subsection{Rotating and slipping solutions}
We will be interested in rotating and slipping solutions of \eqref{z-equation-2}, that is,
$$
z(t,s)=e^{i\Omega t}z_0(s-at),
$$
for some constant angular velocity $\Omega\in\R$. Assuming such ansatz, \eqref{z-equation-2} agrees with
\begin{equation}\label{z-equation-3}
\Omega z_0(\xi)+aiz_0'(\xi)-z_0''(\xi)\pm \frac{2\overline{z_0(\xi)}z_0'(\xi)^2}{1\pm |z_0(\xi)|^2}=0,
\end{equation}
where $\xi=s-at$. We would like to work with the complex notation $w=e^{i\xi}$, instead of the variable $\xi$. Then, denote $g$ as
$$
z_0(\xi)=g(e^{i\xi})=g(w).
$$
In that way, \eqref{z-equation-3} amounts to
\begin{equation}\label{z-equation-4}
\Omega g(w)+(1-a)wg'(w)+w^2g''(w)\mp \frac{2\overline{g(w)}w^2g'(w)^2}{1\pm |g(w)|^2}=0,
\end{equation}
for any $w\in\T$.

If 
\begin{equation}\label{trivial}
\Omega_R=a+\frac{-1\pm R^2}{1\pm R^2}, \quad g(w)=Rw,
\end{equation}
we have that \eqref{z-equation-4} is trivially satisfied. This family of steady solutions correspond to the circle ($R=1$ and $a=0$) and the helix solutions.

In this work, we would like to find nontrivial solutions to \eqref{z-equation-4} that are perturbations of the trivial solutions \eqref{trivial}. For that reason, we will better work with $f$ instead of $g$ defined as 
\begin{equation}\label{f}
g(w)=w(R+f(w)).
\end{equation}
Moreover, we will work with $\lambda$ instead of $\Omega$ defined as
\begin{equation}\label{lambda}
\Omega=\Omega_R+\lambda.
\end{equation}
Hence, \eqref{z-equation-4} agrees with 
$$
G_a(R,\lambda, f)(w)=0, \quad \forall w\in\T,
$$
with
\begin{align}\label{G-eq}
G_a(R,\lambda,f)(w):=&\Omega R+\Omega f(w)+(1-a)R+(1-a)f(w)+(3-a)wf'(w)+w^2f''(w)\\
&\mp 2\frac{(R+\overline{f(w)})(R+f(w)+wf'(w))^2}{1\pm (R^2+Rf(w)+R\overline{f(w)}+f(w)\overline{f(w)})}.\nonumber
\end{align}

In the following proposition, we give the trivial solutions of $G_a$, which corresponds to the helix or the circle.

\begin{pro}\label{prop-trivial}
For any $a\in\R$, we have that $G_a(R,0,0)\equiv 0$, for any $R>0$, with $R\leq 1$ for the Hyperbolic case.
\end{pro}

In the following remark, we explain how to come back to the tangent ${\bf T}$ or the curve ${\bf X}$ once we have a solution for \eqref{G-eq}.
\begin{rem}\label{rem-original}
Once we have a nontrivial solution of \eqref{G-eq}, we have found a nontrivial $z_0$ as
$$
z(t,s)=e^{i\Omega t}z_0(s-at).
$$
We can come back to the original variables ${\bf T}$ and ${\bf X}$ as follows. Using \eqref{equation-T-z}, we find that
$$
\mathbf{T}(t,s)=(T_1,T_2,T_3)=\mathscr{R}^{\Omega t}\, \mathbf{T}^0(s-at),
$$
where $\mathbf{T}^0$ is the initial tangent vector and can be computed also using \eqref{equation-T-z}:
\begin{equation}\label{T0}
\mathbf{T}^0=\frac{1}{1\pm |z_0(s)|^2}\left(2z_0(s),1\mp |z_0(s)|^2\right).
\end{equation}
We define
$$
\mathscr{R}^{\Omega t} \mathbf{x}=(e^{i\Omega t}(x_1,x_2),x_3).
$$
Now, let us come back to $\mathbf{X}$ using that $\partial_s \mathbf{X}=\mathbf{T}$ and $\partial_t \mathbf{X}=\mathbf{T}\wedge_\pm \mathbf{T}_s$. From the last equation one has that
$$
\mathbf{X}=\mu(s)+\int_0^t (\mathbf{T}\wedge_\pm \mathbf{T}_s)(\tau,s)d\tau,
$$
for some $\mu(s)$. By using now $\partial_s \mathbf{X}=\mathbf{T}$ we get
$$
\mu(s)=\mu_0+\int_{s_0}^s \mathbf{T}(t,\sigma)d\sigma-\int_0^t (\mathbf{T}\wedge_\pm \mathbf{T}_s)(\tau,s)d\tau+\int_0^t (\mathbf{T}\wedge_\pm \mathbf{T}_s)(\tau,s_0)d\tau,
$$
for any $s_0$ and $\mu(s_0)=\mu_0\in\R$. Then:
\begin{align*}
\mathbf{X}=&\mu_0+\int_{s_0}^s \mathbf{T}(t,\sigma)d\sigma+\int_0^t (\mathbf{T}\wedge_\pm \mathbf{T}_s)(\tau,s_0)d\tau\\
=&\mu_{0}+\mathscr{R}^{\Omega t}\int_{s_0}^s \mathbf{T}^0(\sigma-at)d\sigma+\int_0^t (\mathbf{T}\wedge_\pm \mathbf{T}_s)(\tau,s_0)d\tau\\
=&\mu_{0}+\mathscr{R}^{\Omega t}\int_{s_0}^s \mathbf{T}^0(\sigma-at)d\sigma+\int_0^t (\mathscr{R}^{\Omega\tau}\mathbf{T^0}(s_0-a\tau)\wedge_\pm \mathscr{R}^{\Omega\tau}\mathbf{T}_s^0(s_0-a\tau))d\tau.
\end{align*}
Note that
$$
(\mathscr{R}^{\Omega\tau}\mathbf{T^0}(s_0-a\tau)\wedge_\pm \mathscr{R}^{\Omega\tau}\mathbf{T}_s^0(s_0-a\tau))=\mathscr{R}^{\Omega\tau}\left(\mathbf{T^0}(s_0-a\tau)\wedge_\pm \mathbf{T}_s^0(s_0-a\tau)\right).
$$
We can always fix (up to a translation) that $s_0=0$ and $X_0(0)=\mu_0=0$ obtaining then
\begin{align*}
\mathbf{X}=&\mathscr{R}^{\Omega t}\int_{0}^s \mathbf{T}^0(\sigma-at)d\sigma+\int_0^t \mathscr{R}^{\Omega\tau}\left(\mathbf{T^0}(-a\tau)\wedge_\pm \mathbf{T}_s^0(-a\tau)\right))d\tau.
\end{align*}

In the special case that $a=0$, the last integral can be computed as:
\begin{align*}
\int_0^t (\mathbf{T}\wedge_\pm \mathbf{T}_s)(\tau,s_0)d\tau=&-\frac{i}{\Omega}\mathscr{R}^{\Omega t}\left(\mathbf{T^0}(s_0)\wedge_\pm \mathbf{T}_s^0(s_0)\right)+\frac{i}{\Omega}\left(\mathbf{T^0}(s_0)\wedge_\pm \mathbf{T}_s^0(s_0)\right)
\\
=&\frac{1}{\Omega}\Big((T_3^0)_s(\cos(\Omega t) T_1^0+\sin(\Omega t) T_2^0)-T_3^0(\cos(\Omega t)(T_1^0)_s+\sin(\Omega t)(T_2^0)_s,   \\
& -(T_3^0)_s(\sin(\Omega t) T_1^0-\cos(\Omega t) T_2^0)+T_3^0(\sin(\Omega t)(T_1^0)_s-\cos(\Omega t)(T_2^0)_s,\\
&\pm\left\{(T^0_1(T^0_2)_s-(T^0_1)_sT^0_2)t\right\} \Big)\\
&-\frac{1}{\Omega}\Big((T_3^0)_s T_1^0-T_3^0(T_1^0)_s, (T_3^0)_s T_2^0-T_3^0(T_2^0)_s,0 \Big).
\end{align*}
In the above, we denote $i \mathscr{R}^{\Omega t} x=(ie^{i\Omega t}(x_1,x_2),x_3)$. Then, we achieve
\begin{align*}
\mathbf{X}(t,s)=&\mathscr{R}^{\Omega t}\int_{0}^s \mathbf{T}^0(\sigma)d\sigma-\frac{i}{\Omega}\mathscr{R}^{\Omega t}\left(\mathbf{T^0}(0)\wedge_\pm \mathbf{T}_s^0(0)\right)+\frac{i}{\Omega}\left(\mathbf{T^0}(0)\wedge_\pm \mathbf{T}_s^0(0)\right),
\end{align*}
where let us recall that ${\bf T}^0$ is totally described through \eqref{T0}.
\end{rem}

\subsection{Functional analytical setting}
We will consider $f\in C^2$ lying in the following space
\begin{equation}\label{X}
X:=\left\{f\in C^2(\T), \quad f(w)=\sum_{n\in \Z, n\neq 0}f_n w^n, f_n\in\R\right\}.
\end{equation}
The condition $f_n\in\R$ agrees with the symmetry of $f$ with respect to the $x$ axis:
$$
f(\overline{w})=\overline{f(w)}.
$$ 

\begin{rem}
We should remove the mode $n=0$ in the definition of $f$ in order to obtain nontrivial solutions of $G_a$. Note that such a mode only gives us a dilatation of the trivial solution $f(w)=0$.
\end{rem}

We define the range space as
\begin{equation}\label{Y}
Y:=\left\{f\in C^0(\T), \quad f(w)=\sum_{n\in \Z}f_n w^n, f_n\in\R\right\}.
\end{equation}

The space for values of $R$ is
\begin{equation}
\mathcal{R}:=\left\{ 
\begin{array}{l}
\R^+, \quad \textnormal{euclidean case},\\ 
(0,1), \quad \textnormal{hyperbolic case}.
\end{array}
\right.
\end{equation}
Then, the nonlinear operator \eqref{G-eq} is well-defined and $C^1$ is such spaces:

\begin{pro}\label{prop-well-defined-1}
The functional $G_a:\mathcal{R}\times \R\times X\rightarrow Y$ is well-defined and $C^1$.
\end{pro}
\begin{proof}
Due to the expression of $G_a$, it is clear that it is $C^0$ if $f\in C^2$. Let us check the symmetry persistence. Take $f\in Y$, then we have that $f(\overline{w})=\overline{f(w)}$. We wish to prove that
$$
G_a(R,\lambda,f)(\overline{w})=\overline{G_a(R,\lambda,f)(w)},
$$
which comes easily from the expression of $G_a$ in \eqref{G-eq} taking into account that $\Omega, R\in\R$.
\end{proof}

\subsection{Crandall-Rabinowitz theorem}\label{sec-CR}
In this section, we recall the classical Crandall-Rabinowitz Theorem whose proof can be found  in \cite{CrandallRabinowitz}.

\begin{theo}[Crandall-Rabinowitz Theorem]\label{CR}
    Let $X, Y$ be two Banach spaces, $V$ be a neighborhood of $0$ in $X$ and $F:\mathbb{R}\times V\rightarrow Y$ be a function with the properties,
    \begin{enumerate}
        \item $F(\lambda,0)=0$ for all $\lambda\in\mathbb{R}$.
        \item The partial derivatives  $\partial_\lambda F_{\lambda}$, $\partial_fF$ and  $\partial_{\lambda}\partial_fF$ exist and are continuous.
        \item The operator $\partial_f F(\lambda_0,0)$ is Fredholm of zero index and $\textnormal{Ker}(F_f(\lambda_0,0))=\langle f_0\rangle$ is one-dimensional. 
                \item  Transversality assumption: $\partial_{\lambda}\partial_fF(\lambda_0,0)f_0 \notin \textnormal{Im}(\partial_fF(\lambda_0,0))$.
    \end{enumerate}
    If $Z$ is any complement of  $\textnormal{Ker}(\partial_fF(\lambda_0,0))$ in $X$, then there is a neighborhood  $U$ of $(\lambda_0,0)$ in $\mathbb{R}\times X$, an interval  $(-a,a)$, and two continuous functions $\Phi:(-a,a)\rightarrow\mathbb{R}$, $\beta:(-a,a)\rightarrow Z$ such that $\Phi(0)=\lambda_0$ and $\beta(0)=0$ and
    $$F^{-1}(0)\cap U=\{(\Phi(s), s f_0+s\beta(s)) : |s|<a\}\cup\{(\lambda,0): (\lambda,0)\in U\}.$$
\end{theo}

In this context, we will say that $\lambda_0$ is an eigenvalue of $F$. However, we can not a priori apply the above theorem to $G$ since we have an extra parameter given by $\Omega$. However, we can use the following modification of the theorem:

\begin{theo}[Crandall-Rabinowitz Theorem with parameters]
    Let $X, Y$ be two Banach spaces, $V$ be a neighborhood of $0$ in $X$ and $F:\mathbb{R}\times\R \times V\rightarrow Y$ be a function with the properties,
    \begin{enumerate}
        \item $F(\lambda,0,0)=0$ for all $\lambda\in\mathbb{R}$.
        \item The partial derivatives  $\partial_\lambda F_{\lambda}$, $\partial_fF$, $\partial_\mu F$, $\partial_{\lambda}\partial_\mu F$ and  $\partial_{\lambda}\partial_fF$ exist and are continuous.
        \item The operator $\partial_{(\mu,f)} F(\lambda_0,0,0)$ is Fredholm of zero index and $\textnormal{Ker}(F_{(\mu,f)}(\lambda_0,0,0))=\langle f_0\rangle$ is one-dimensional. 
                \item  Transversality assumption: $\partial_{\lambda}\partial_{(\mu,f)}F(\lambda_0,0,0)f_0 \notin \textnormal{Im}(\partial_{(\mu,f)}F(\lambda_0,0,0))$.
    \end{enumerate}
    If $Z$ is any complement of  $\textnormal{Ker}(\partial_{(\mu,f)}F(\lambda_00,,0))$ in $X$, then there is a neighborhood  $U$ of $(\lambda_0,0,0)$ in $\mathbb{R}\times\R\times X$, an interval  $(-a,a)$, and two continuous functions $\Phi:(-a,a)\rightarrow\mathbb{R}$, $\beta:(-a,a)\rightarrow \R\times Z$ such that $\Phi(0)=\lambda_0$ and $\beta(0)=(0,0)$ and
    $$F^{-1}(0)\cap U=\{(\Phi(s), s f_0+s\beta(s)) : |s|<a\}\cup\{(\lambda,0,0): (\lambda,0)\in U\}.$$
\end{theo}

\section{Spectral study for only rotation motion}\label{sec-spectral-0} 
This section deals with the spectral study of the linearized operator of $G_0$ at the trivial solution. In order to apply the Crandall-Rabinowitz theorem, we should find some eigenvalues $R$ such that the kernel of the linearized operator is not trivial. First, we will give different expressions of such linear operator and check that it is a Fredholm operator of zero index. Later, we will study its kernel finding the appropriate eigenvalues. Finally, we will study the range in order to verify the transversal condition at the end.

For the sake of simplicity, here we will give the details for only  rotation motion, that is, $a=0$. The computations for $a\neq 0$ will be developed in the following section.

\begin{pro}\label{prop-linop-1}
The linearized operator of $G_0$ around $(R,0,0)$ reads as 
\begin{align*}
\partial_f G_0(R,0,0)h(w)=\mp 4\frac{R^2}{(1\pm R^2)^2}\textnormal{Re}\left[h(w)\right]+\frac{3\mp R^2}{1\pm R^2} wh'(w)+w^2h''(w).
\end{align*}
\end{pro}
\begin{proof}
From the expression of $G_0$ in \eqref{G-eq}, we find
\begin{align*}
\partial_f G_0(R,0,0)h(w)=&(\Omega_R+1) h(w)+3wh'(w)+w^2h''(w)\\
&\mp 2 \frac{R^2  }{1\pm R^2}\overline{h(w)}\mp 4\frac{R^2}{1\pm R^2}(h(w)+wh'(w))\\
&+ 2\frac{R^4}{(1\pm R^2)^2}(h(w)+\overline{h(w)}).
\end{align*}
Using the expression of $\Omega_R$ with $a=0$, we find
\begin{align*}
\partial_f G_0(R,0,0)h(w)=& \pm\frac{2R^2}{1\pm R^2} h(w)+3wh'(w)+w^2h''(w)\\
&\mp 2 \frac{R^2  }{1\pm R^2}\overline{h(w)}\mp 4\frac{R^2}{1\pm R^2}(h(w)+wh'(w))\\
&+ 2\frac{R^4}{(1\pm R^2)^2}(h(w)+\overline{h(w)})\\
=&\mp 4\frac{R^2}{1\pm R^2}\textnormal{Re}\left[h(w)\right]+4\frac{R^4}{(1\pm R^2)^2}\textnormal{Re}\left[h(w)\right]\\
&+\left(3\mp 4\frac{R^2}{1\pm R^2}\right)wh'(w)+w^2h''(w)\\
=&\mp 4\frac{R^2}{(1\pm R^2)^2}\textnormal{Re}\left[h(w)\right]+\frac{3\mp R^2}{1\pm R^2} wh'(w)+w^2h''(w).
\end{align*}
\end{proof}

In the following, we check that the linear operator is a compact perturbation of an isomorphism, and hence it is a Fredholm operator of zero index.
\begin{pro}\label{prop-fredholm}
The linearized operator $\partial_\lambda G_0(R,0,0)\mu+\partial_f G_0(R,0,0)h$ is Fredholm of zero index.
\end{pro}
\begin{proof}
From Proposition \ref{prop-linop-1} we can write $\partial_\lambda G_0(R,0,0)\mu+\partial_f G_0(R,0,0)h$ as
$$
\partial_\lambda G_0(R,0,0)\mu+\partial_f G_0(R,0,0)h=\mathcal{L}(\mu,h)+\mathcal{K}h,
$$
where
\begin{align*}
\mathcal{L}(\mu,h)(w)=&R\mu+w^2h''(w)+3wh'(w),\\
\mathcal{K}h(w)=&\mp 4\frac{R^2}{(1\pm R^2)^2}\textnormal{Re}\left[h(w)\right]+\frac{3\mp R^2}{1\pm R^2} wh'(w)-3wh'(w).
\end{align*}
Note that $\mathcal{K}h\in C^1$, and since $C^1\subset C^0$ is a compact embedding, we get that $\mathcal{K}$ is a compact operator from $C^2$ to $C^0$. Let us check that $\mathcal{L}:\R\times X\rightarrow Y$ is an isomorphism. Take $d\in Y$ in Fourier series as
$$
d(w)=\sum_{n\in\Z}d_nw^n,
$$
and $h\in X$ as
$$
h(w)=\sum_{n\in\Z }h_n w^n.
$$
Hence $\mathcal{L}(\mu,h)=d$ agrees with
$$
R\mu+\sum_{n\geq 0}h_n w^n n(n+2)=\sum_{n\in\Z} d_nw^n.
$$
In that way, we find the solution
$$
\mu=\frac{d_0}{R},\quad h_n=\frac{d_n}{n(n+2)},\,  n\neq 0,
$$
that is,
$$
h(w)=\sum_{n\neq 0}\frac{d_n}{n(n+2)}w^n.
$$
Let us check that $h\in C^2$ and hence we obtain that $h\in X$. Note that
$$
h''(w)=\sum_{n\neq 0}\frac{n-1}{n+2} d_n w^n=\sum_{n\neq 0}d_n w^n-3\sum_{n\neq 0}\frac{1}{n+2} d_n w^n.
$$
Note that the first term is continuous since $d\in C^0$, and the second term can be seen as a convolution:
$$
\sum_{n\neq 0}\frac{1}{n+2} d_n w^n=d\star \sum_{n\neq 0}\frac{1}{n+2}w^n.
$$
By Parseval's identity we have that the r.h.s term is in $L^2$, and in particular, it is in $L^1$, whereas $d\in C^0$, implying that $w\mapsto \sum_{n\neq 0}\frac{1}{n+2} d_n w^n\in C^0$. Hence $\mathcal{L}:\R\times X\rightarrow Y$ is an isomorphism, and then it is Fredholm of zero index. Since compact perturbations of Fredholm operators remain Fredholm of same index, we can conclude that $\partial_\lambda G_0(R,0,0)\mu+\partial_f G_0(R,0,0)h$ is Fredholm of zero index.
\end{proof}

Finally, let us give the expression of the linearized operator in Fourier series.
\begin{pro}\label{prop-linop-2}
If $h$ takes the form
\begin{equation}\label{h-expression}
h(w)=\sum_{n\in\Z, n\neq 0}a_n w^n=\sum_{n\geq 1}\left\{a_n w^n+a_{-n}\overline{w}^n\right\},
\end{equation}
then
\begin{align*}
\partial_f G_0(R,0,0)h(w)=&\sum_{n\geq 1}\left[\left(\mp 4\frac{R^2}{(1\pm R^2)^2}+n^2\right)(a_n+a_{-n})+2\frac{1\mp R^2}{1\pm R^2}n(a_n-a_{-n})\right]\cos(n\theta)\\
&+i\left[2\frac{1\mp R^2}{1\pm R^2}n(a_n+a_{-n})+n^2(a_n-a_{-n})\right]\sin(n\theta).
\end{align*}
\end{pro}
\begin{proof}
Note that if $h$ satisfies \eqref{h-expression}, we have
\begin{align*}
wh'(w)=&\sum_{n\geq 1}\left\{n(a_n-a_{-n})\cos(n\theta)+in(a_n+a_{-n})\sin(n\theta)\right\},\\
w^2h''(w)=&\sum_{n\geq 1}\left\{[n^2(a_n+a_{-n})-n(a_n-a_{-n})]\cos(n\theta)+i[n^2(a_n-a_{-n})-n(a_n+a_{-n})]\sin(n\theta)\right\}.
\end{align*}

Using the expression of $\partial_f G(R,0,0)$ in Proposition \ref{prop-linop-1} we find
\begin{align*}
\partial_f G_0(R,0,0)h(w)=&\sum_{n\geq 1}\left[\left(\mp 4\frac{R^2}{(1\pm R^2)^2}+n^2\right)(a_n+a_{-n})+\left(\frac{3\mp R^2}{1\pm R^2}-1\right)n(a_n-a_{-n})\right]\cos(n\theta)\\
&+i\left[\left(\frac{3\mp R^2}{1\pm R^2}-1\right)n(a_n+a_{-n})+n^2(a_n-a_{-n})\right]\sin(n\theta)\\
=&\sum_{n\geq 1}\left[\left(\mp 4\frac{R^2}{(1\pm R^2)^2}+n^2\right)(a_n+a_{-n})+2\frac{1\mp R^2}{1\pm R^2}n(a_n-a_{-n})\right]\cos(n\theta)\\
&+i\left[2\frac{1\mp R^2}{1\pm R^2}n(a_n+a_{-n})+n^2(a_n-a_{-n})\right]\sin(n\theta).
\end{align*}
\end{proof}

\subsection{Properties of the linearized operator}
This section aims to give the necessary properties of the linearized operator $\partial_{(\lambda,f)}G_0(R,0,0)$ in order to apply Crandall-Rabinowitz theorem. First, we will study its kernel. For that, we need this preliminary lemma.
\begin{lem}\label{lem-kernel}
The system 
\begin{equation}\label{kernel-system}
\left(
\begin{array}{cc}
\mp 4\frac{ R^2}{(1\pm R^2)^2}+n^2 & 2\frac{1\mp R^2}{1\pm R^2}n\\
2\frac{1\mp R^2}{1\pm R^2}n & n^2
\end{array}
\right)
\left(
\begin{array}{l}
a_n+a_{-n}\\ a_n-a_{-n}
\end{array}
\right)
=
\left(
\begin{array}{l}
0\\ 0
\end{array}
\right), \quad n\geq 1,
\end{equation}
satisfying the following.
\begin{itemize}
\item (Euclidean case) If $R=R_1^e:=1$, then the system \eqref{kernel-system} has a nontrivial solution given by
$$
a_n=a_{-n}=0, \, n\geq 2, \quad a_1=a_{-1}=\kappa,\, \kappa\in\R.
$$
Otherwise, the only solution to \eqref{kernel-system} is the trivial one.
\item (Hyperbolic case) If 
\begin{equation}\label{kernel-R-n-hyp-3-prop}
R=R_N^h:=\sqrt{\frac{ \left(N^2+2\right)- 2\sqrt{3}\sqrt{N^2-1}}{N^2-4}},\quad N\geq 3,
\end{equation}
then
\begin{align}\label{kernel-solution}
a_n=&a_{-n}=0,\nonumber \quad n\neq N, \\ a_N=&\frac{N(1-(R_N^h)^2)-2(1+(R_N^h)^2)}{N(1-(R_N^h)^2)+2(1+(R_N^h)^2)}a_{-N}=:\alpha_N a_{-N}, \quad a_{-N}=\kappa\in\R.
\end{align}
Otherwise, the only solution to \eqref{kernel-system} is the trivial one.
\end{itemize}
\end{lem}
\begin{proof}
Denote now $b_n=a_n+a_{-n}$ and $c_n=a_n-a_{-n}$. Trivially, we have that 
$$a_n=\frac{b_n+c_n}{2},\quad a_{-n}=\frac{b_n-c_n}{2}.$$
In this way, \eqref{kernel-system-1} reads as
system for each mode $n\geq 1$:
\begin{equation}\label{kernel-system-2}
\left(
\begin{array}{cc}
\mp 4\frac{ R^2}{(1\pm R^2)^2}+n^2 & 2\frac{1\mp R^2}{1\pm R^2}n\\
2\frac{1\mp R^2}{1\pm R^2}n & n^2
\end{array}
\right)
\left(
\begin{array}{l}
b_n\\ c_n
\end{array}
\right)
=
\left(
\begin{array}{l}
0\\ 0
\end{array}
\right).
\end{equation}
Hence, if $b_n=c_n=0$, then $a_n=a_{-n}$ and the kernel is trivial. In order to find nontrivial solutions (note that we need that the dimension is one in order to implement the Crandall-Rabinowitz theorem) we need to solve the equation
\begin{equation*}
\left|
\begin{array}{cc}
\mp 4\frac{ R^2}{(1\pm R^2)^2}+n^2 & 2\frac{1\mp R^2}{1\pm R^2}n\\
2\frac{1\mp R^2}{1\pm R^2}n & n^2
\end{array}
\right|=0,
\end{equation*}
and that is
\begin{equation*}
\mp 4\frac{ R^2}{(1\pm R^2)^2}+n^2-4\frac{(1\mp R^2)^2}{(1\pm R^2)^2}=0,
\end{equation*}
agreeing with
\begin{equation}\label{kernel-det}
n^2=4\frac{(1\mp R^2)^2\pm R^2}{(1\pm R^2)^2}=4\frac{1+R^4\mp R^2}{(1\pm R^2)^2}.
\end{equation}
We can also solve \eqref{kernel-det} in terms of $n$. For the Euclidean case, we find two solutions given by
\begin{equation}\label{kernel-R-n-euclidean}
R^2=\frac{- \left(n^2+2\right)\pm 2\sqrt{3}\sqrt{n^2-1}}{n^2-4}.
\end{equation}
whereas for the Hyperbolic case we achieve
\begin{equation}\label{kernel-R-n-hyp}
R^2=\frac{ \left(n^2+2\right)\pm 2\sqrt{3}\sqrt{n^2-1}}{n^2-4}.
\end{equation}

Note that here we have a constrain since $R_n^2>0$. For $n=1$, we find
$$
R^2=\pm 1,
$$
where $\pm$ refers to the Euclidean and Hyperbolic cases, respectively. Then, the only valid solution is for the Euclidean case: $n=1$ and $R_1^{e}:=1$. In such a case, we find that the solutions of \eqref{kernel-system-1} are $a_1=a_{-1}=\kappa$, for $\kappa\in\R$.

In the case of $n=2$, we have a singularity in the denominator of \eqref{kernel-R-n-euclidean} and \eqref{kernel-R-n-hyp}. However, we can obtain the limit as
$$
R^2=0,
$$
in each case, and we will not consider it since $R^2>0$.

For $n\geq 3$, we do not find any positive solution of \eqref{kernel-R-n-euclidean} for the Euclidean case. However, we do find two solutions for the Hyperbolic case:
\begin{equation}\label{kernel-R-n-hyp-2}
R=\sqrt{\frac{ \left(n^2+2\right)\pm 2\sqrt{3}\sqrt{n^2-1}}{n^2-4}}.
\end{equation}
In such a case we have an extra condition that is $R<1$ and then the only solution is 
\begin{equation}\label{kernel-R-n-hyp-3}
R_n^h:=\sqrt{\frac{ \left(n^2+2\right)- 2\sqrt{3}\sqrt{n^2-1}}{n^2-4}},\quad n\geq 3.
\end{equation}
Note also that $R_n^h$ increases to $1$. By fixing $n$ and $R=R_n^h$ in \eqref{kernel-system-2} we find the solution
$$
c_n=-\frac{2}{n}\frac{1+(R_n^h)^2}{1-(R_n^h)^2} b_n, \quad b_n\in\R.
$$
%
Coming back to the original variables $a_n$ and $a_{-n}$ 
\begin{align*}
a_n=&b_n\frac{n(1-(R_n^h)^2)-2(1+(R_n^h)^2)}{2n(1-(R_n^h)^2)},\\
a_{-n}=&b_n\frac{n(1-(R_n^h)^2)+2(1+(R_n^h)^2)}{2n(1-(R_n^h)^2)}.
\end{align*}
By choosing appropriate $b_n$ we finally get
$$
a_n= \frac{n(1-(R_n^h)^2)-2(1+(R_n^h)^2)}{n(1-(R_n^h)^2)+2(1+(R_n^h)^2)}a_{-n}.
$$
%
%
\end{proof}

In the following proposition, we find the eigenvalues associated to our problem. We observe different situations depending on the euclidean or hiperbolic case. In the first one, we find a unique eigenvalue given by $R=1$ and $n=1$, which correspond to the circle filament. However, the family of eigenvalues in the second case is richer. There, we find a sequence of eigenvalues $R_n$, with $n\geq 3$, which increases to $1$ (which is the limit point).

\begin{pro}\label{prop-kernel}
The following assertions hold true.
\begin{itemize}
\item (Euclidean case) If $R=R_1^e=1$, then 
\begin{align*}
\textnormal{Ker}\left[\partial_{(\lambda,f)} G_0(R_1^e,0,0)\right]=\left\{\kappa(w+\overline{w}),\quad \kappa\in\R\right\}.
\end{align*}
Otherwise, the kernel is trivial.
\item (Hyperbolic case) If $R=R_N^h$, for fixed $N\geq 3$, then
\begin{align*}
\textnormal{Ker}\left[\partial_{(\lambda,f)} G_0(R_N^h,0,0)\right]=\left\{\kappa(\alpha_N w^N+\overline{w}^N),\quad \kappa\in\R\right\},
\end{align*}
where $R_N^h$ and $\alpha_N$ are defined in Lemma \ref{lem-kernel}. Otherwise, the kernel is trivial.
\end{itemize}
\end{pro}
\begin{proof}
In order to check the kernel of $\textnormal{Ker}\left[\partial_{(\lambda,f)} G(R,0,0)\right]$, we should study the following equation 
$$
\partial_\lambda G_0(R,0,0)\mu+\partial_f G_0(R,0,0)h=0.
$$
We trivially find that $\mu=0$. Then, by virtue of Proposition \ref{prop-linop-2}, it amounts to study the following system for each mode $n\geq 1$:
\begin{equation}\label{kernel-system-1}
\left(
\begin{array}{cc}
\mp 4\frac{ R^2}{(1\pm R^2)^2}+n^2 & 2\frac{1\mp R^2}{1\pm R^2}n\\
2\frac{1\mp R^2}{1\pm R^2}n & n^2
\end{array}
\right)
\left(
\begin{array}{l}
a_n+a_{-n}\\ a_n-a_{-n}
\end{array}
\right)
=
\left(
\begin{array}{l}
0\\ 0
\end{array}
\right).
\end{equation}
Note that such system was studied in Lemma \ref{lem-kernel}. With the assertions of Lemma \ref{lem-kernel} we conclude the proof.
\end{proof}

In what follows, we aim to characterize the range of the linearized operator. It will be useful for later check that the transversal condition is satisfied in our problem.

\begin{pro}\label{prop-range}
The following assertions hold true.
\begin{itemize}
\item (Euclidean case) If $R=R_1^e=1$, then 
\begin{align}\label{range-euclidean}
\textnormal{Range}\left[\partial_{(\lambda,f)} G_0(R_1^e,0,0)\right]=\left\{f\in Y, \quad f_1+f_{-1}=0\right\}.
\end{align}
Otherwise, the range is $Y$.
\item (Hyperbolic case) If $R=R_N^h$, for fixed $N\geq 3$, then
\begin{align*}
\textnormal{Range}\left[\partial_{(\lambda,f)} G_0(R_N^h,0,0)\right]=\left\{f\in Y, \quad f_{-N}=\frac{\sqrt{1+(R_m^h)^4+(R_m^h)^2}}{1-(R_m^h)^2}f_N\right\},
\end{align*}
where $R_N^h$ is defined in Lemma \ref{lem-kernel}. Otherwise, the range is $Y$.
\end{itemize}
\end{pro}
\begin{proof}
Since the linear operator is a Fredholm operator of zero index by Proposition \ref{prop-fredholm}, we know that the dimension of the kernel equals to the codimension of the range. Hence, as a consequence of Proposition \ref{prop-kernel} we have that the codimension of the range is $1$.  In order to check the range of $\partial_{(\lambda,f)} G_0(R,0,0)$, we should study the following equation 
$$
\partial_\lambda G_0(R,0,0)\mu+\partial_f G_0(R,0,0)h=d,
$$
for $d\in Y$. We can write $d$ in Fourier series as
$$
d(w)=\sum_{n\in\Z}d_n w^n.
$$
Then, we trivially find that $\mu=d_0$. By virtue of Proposition \ref{prop-linop-2}, we have to study the following system
\begin{equation}\label{range-system-1}
\left(
\begin{array}{cc}
\mp 4\frac{ R^2}{(1\pm R^2)^2}+n^2 & 2\frac{1\mp R^2}{1\pm R^2}n\\
2\frac{1\mp R^2}{1\pm R^2}n & n^2
\end{array}
\right)
\left(
\begin{array}{l}
a_n+a_{-n}\\ a_n-a_{-n}
\end{array}
\right)
=
\left(
\begin{array}{l}
d_n+d_{-n}\\ d_n-d_{-n}
\end{array}
\right).
\end{equation}
Note that such matrix has been studied in Lemma \ref{lem-kernel}. 

In the Euclidean case, let us fix $R=R^e_1=1$. Then, by Lemma \ref{lem-kernel} we have that \eqref{range-system-1} has a solution for $n\neq 1$. However, for $n=1$ the system \eqref{range-system-1} has not always a solution. Note that this case correspond to 
\begin{equation*}
\left(
\begin{array}{cc}
0 & 0\\
0 & 1
\end{array}
\right)
\left(
\begin{array}{l}
a_1+a_{-1}\\ a_1-a_{-1}
\end{array}
\right)
=
\left(
\begin{array}{l}
d_1+d_{-1}\\ d_1-d_{-1}
\end{array}
\right),
\end{equation*}
which has a solution only if $d_1+d_{-1}=0$. That concludes the proof of \eqref{range-euclidean}.

Let us now work with the Hyperbolic case. Fix $N\geq 3$ and $R=R_N^h$ defined in Lemma \ref{lem-kernel}. By the equivalence between \eqref{kernel-det} and \eqref{kernel-R-n-hyp}, we have that $R=R_N^h$ agrees with
$$
N^2=4\frac{1+R^4+R^2}{(1-R^2)^2}.
$$
In such a case, the system \eqref{range-system-1} amounts to
\begin{equation}\label{range-system-hyp}
\left(
\begin{array}{cc}
4\frac{ (1+R^2)^2}{(1- R^2)^2} & 4\frac{(1+ R^2)\sqrt{1+R^4+R^2})}{(1- R^2)^2}\\
4\frac{(1+ R^2)\sqrt{1+R^4+R^2}}{(1- R^2)^2} & 4\frac{1+R^4+R^2}{(1-R^2)^2}
\end{array}
\right)
\left(
\begin{array}{l}
a_n+a_{-n}\\ a_n-a_{-n}
\end{array}
\right)
=
\left(
\begin{array}{l}
d_n+d_{-n}\\ d_n-d_{-n}
\end{array}
\right).
\end{equation}
If $n\neq N$, then \eqref{range-system-hyp} has a unique solution, for any $d_n, d_{-n}\in\R$. In the case $n=N$, the system has a solution only if
$$
d_{-N}=\frac{\sqrt{1+R^4+R^2}}{1-R^2} d_N.
$$

\end{proof}

\subsection{Transversal condition}
Finally, we check the transversal condition of the Crandall-Rabinowitz theorem.
\begin{pro}\label{prop-transversal}
The transversal condition is satisfied, that is,
$$
\partial_R \partial_{(\lambda,f)} G_0(R^*,0,0)(\lambda,h)^*\notin \textnormal{Range}\left[\partial_{(\lambda,f)}G_0(R^*,0,0)\right],
$$
where $R^*$ equals to $R_1^e$ for the Euclidean case, and $R^*$ equals to $R_N^h$ in the Hyperbolic case. Here, $h^*$ is an element of $\textnormal{Ker}\left[\partial_{(\lambda,f)}G_0(R^*,0,0)\right]$.
\end{pro}

\begin{proof}
From Proposition \ref{prop-linop-2} we have
\begin{align*}
\partial_R \partial_f G_0(R,0,0)h(w)=&\sum_{n\geq 1}\left[\mp 8R\frac{1\mp R^3}{(1\pm R^2)^3} (a_n+a_{-n})\mp 8 \frac{R}{(1\pm R^2)^2} n(a_n-a_{-n})\right]\cos(n\theta)\\
&+i\left[\mp 8 \frac{R}{(1\pm R^2)^2}n(a_n+a_{-n})\right]\sin(n\theta).
\end{align*}
Let us first work with the Euclidean case. Then $R^*=R_1^e=1$, and the element of the kernel satisfies $a_n=a_{-n}$ by Proposition \ref{prop-kernel}. Then:
\begin{align*}
\partial_R \partial_{(\lambda,f)} G_0(R^*,0,0)h^*(w)=&\left[- 16 R^*\frac{1+ (R^*)^3}{(1+ (R^*)^2)^3} a_1\right]\cos(\theta)+i\left[- 16  \frac{R^*}{(1+ (R^*)^2)^2} a_1\right]\sin(\theta)\\
=&\left[- 16 \frac{1}{4} a_1\right]\cos(\theta)+i\left[- 16  \frac{1}{4}1 a_n\right]\sin(\theta)\\
=&- 16 \frac{1}{4} a_1 w.
\end{align*}
By the characterization of $\textnormal{Range}\left[\partial_{(\lambda,f)}G_0(R^*,0,0)\right]$ in Proposition \ref{prop-range}, we get that $$w\notin \textnormal{Range}\left[\partial_{(\lambda,f)}G_0(R^*,0,0)\right],$$
and hence the transversal condition is satisfied in this case.

Finally, let us work with the Hyperbolic case. Set $N\geq 3$ and $R^*=R_N^h$, where $R_N^h$ is defined in Lemma \ref{lem-kernel}. Then, using that a element of the kernel verifies $a_{-N}=1$ and $a_N=\alpha_N$ by Proposition \ref{prop-kernel}, we get
\begin{align*}
\partial_R \partial_{(\lambda,f)} G_0(R^\star,0,0)h^\star(w)=&\left[ 8R^\star\frac{1+ (R^\star)^3}{(1- (R^\star)^2)^3} (\alpha_N+1)+ 8 \frac{R^\star}{(1- (R^\star)^2)^2} N(\alpha_N-1)\right]\cos(N\theta)\\
&+i\left[ 8 \frac{R^\star}{(1- (R^\star)^2)^2}N(\alpha_N+1)\right]\sin(N\theta)\\
=:& c_N\cos(N\theta)+i d_N\sin(N\theta).
\end{align*}
Define $g_N$ and $g_{-N}$ as
$$
g_N+g_{-N}=c_N, \quad g_N-g_{-N}=d_N,
$$
which agrees with
$$
g_N=\frac{c_N+d_N}{2},\quad g_{-N}=\frac{c_N-d_N}{2}.
$$
Then, by Proposition \ref{prop-range}, we have that $d_N\cos(N\theta)+i c_N\sin(N\theta)\notin \textnormal{Range}\left[\partial_{(\lambda,f)}G_0(R^*,0,0)\right]$ if 
$$
g_{-N}\neq \frac{\sqrt{1+(R^\star)^4+(R^2)^\star}}{1-(R^\star)^2}g_N,
$$
agreeing with 
$$
(c_N-d_N)\neq \frac{\sqrt{1+(R^\star)^4+(R^2)^\star}}{1-(R^\star)^2}(c_n+d_N).
$$
Using the definition of $c_N$ and $d_N$, such condition amounts to
\begin{align*}
&8R^\star\frac{1+ (R^\star)^3}{(1- (R^\star)^2)^3} (\alpha_N+1)- 16 \frac{R^\star}{(1- (R^\star)^2)^2} N\\
\neq& \frac{\sqrt{1+(R^\star)^4+(R^2)^\star}}{1-(R^\star)^2}\left[ 8R^\star\frac{1+ (R^\star)^3}{(1- (R^\star)^2)^3} (\alpha_N+1)+ 16 \frac{R^\star}{(1- (R^\star)^2)^2} N\alpha_N\right],
\end{align*}
which is satisfied for any $N\geq 3$. Using the definition of $\alpha_N$ in Lemma \ref{lem-kernel} we get that the previous condition amounts to
\begin{align}\label{transv-1}
\nonumber& \frac{(1+ (R^\star)^3)}{N(1-(R^\star)^2)+2(1+(R^\star)^2)}-  1 \\
\neq& \frac{\sqrt{1+(R^\star)^4+(R^2)^\star}}{1-(R^\star)^2}\left[  \frac{(1+ (R^\star)^3)}{N(1-(R^\star)^2)+2(1+(R^\star)^2)}\right.\\
&\left.+   \frac{N(1-(R^\star)^2)-2(1+(R^\star)^2)}{N(1-(R^\star)^2)+2(1+(R^\star)^2)}\right].
\end{align}
We can prove that the left hand side is negative and let us check that the right hand side is positive obtaining then that \eqref{transv-1} is satisfied. The r.h.s. is positive if and only if
$$
1+(R^\star)^3+(N-2)-(N+2)(R^\star)^2>0.
$$
Note that 
$$
1+(R^\star)^3+(N-2)-(N+2)(R^\star)^2>1+(N-2)-(N+2)(R^\star)^2,
$$
which is positive if
$$
\frac{N-1}{N+2}>(R^\star)^2=\frac{N^2+2-2\sqrt{3}\sqrt{N^2-1}}{N^2-4},
$$
for $N\geq 3$. That can be written as
$$
(N-1)(N-2)>N^2+2-2\sqrt{3}\sqrt{N^2-1},
$$
which is true for any $N\geq 3$.
\end{proof}

\section{Spectral study for rotation and slipping motion} \label{sec-spectral-a}
This section deals with the spectral study of the linearized operator of $G_a$ at the trivial solution when $a\neq 0$. Here, the family of eigenvalues is richer, specially in the Euclidean case.

\begin{pro}\label{prop-linop-1-a}
The linearized operator of $G_a$ around $(R,0,0)$ reads as 
\begin{align*}
\partial_f G_a(R,0,0)h(w)=\mp 4\frac{R^2}{(1\pm R^2)^2}\textnormal{Re}\left[h(w)\right]+\frac{3\mp R^2}{1\pm R^2} wh'(w)-awh'(w)+w^2h''(w).
\end{align*}
\end{pro}
\begin{proof}
From the expression of $G$ in \eqref{G-eq}, we find
\begin{align*}
\partial_f G_a(R,0,0)h(w)=&(\Omega_R+1-a) h(w)+(3-a)wh'(w)+w^2h''(w)\\
&\mp 2 \frac{R^2  }{1\pm R^2}\overline{h(w)}\mp 4\frac{R^2}{1\pm R^2}(h(w)+wh'(w))\\
&+ 2\frac{R^4}{(1\pm R^2)^2}(h(w)+\overline{h(w)}).
\end{align*}
Using the expression of $\Omega_R$, we find
\begin{align*}
\partial_f G_a(R,0,0)h(w)=& \pm\frac{2R^2}{1\pm R^2} h(w)+(3-a)wh'(w)+w^2h''(w)\\
&\mp 2 \frac{R^2  }{1\pm R^2}\overline{h(w)}\mp 4\frac{R^2}{1\pm R^2}(h(w)+wh'(w))\\
&+ 2\frac{R^4}{(1\pm R^2)^2}(h(w)+\overline{h(w)})\\
=&\mp 4\frac{R^2}{1\pm R^2}\textnormal{Re}\left[h(w)\right]+4\frac{R^4}{(1\pm R^2)^2}\textnormal{Re}\left[h(w)\right]\\
&+\left(3\mp 4\frac{R^2}{1\pm R^2}\right)wh'(w)-awh'(w)+w^2h''(w)\\
=&\mp 4\frac{R^2}{(1\pm R^2)^2}\textnormal{Re}\left[h(w)\right]+\frac{3\mp R^2}{1\pm R^2} wh'(w)-awh'(w)+w^2h''(w).
\end{align*}
\end{proof}

In the next proposition, we show the fredholmness of the linearized operator at the equilibrium. We omit the proof due to its similarly with Proposition \ref{prop-fredholm}.
\begin{pro}\label{prop-fredholm-a}
The linearized operator $\partial_\lambda G_a(R,0,0)\mu+\partial_f G_a(R,0,0)h$ is Fredholm of zero index.
\end{pro}

Finally, let us give the expression of the linear operator in Fourier series.
\begin{pro}\label{prop-linop-2-a}
If $h$ takes the form
\begin{equation}\label{h-expression-a}
h(w)=\sum_{n\in\Z, n\neq 0}a_n w^n=\sum_{n\geq 1}\left\{a_n w^n+a_{-n}\overline{w}^n\right\},
\end{equation}
then
\begin{align*}
\partial_f G_a(R,0,0)h(w)=&\sum_{n\geq 1}\left[\left(\mp 4\frac{R^2}{(1\pm R^2)^2}+n^2\right)(a_n+a_{-n})+\left(2\frac{1\mp R^2}{1\pm R^2}-a\right)n(a_n-a_{-n})\right]\cos(n\theta)\\
&+i\left[\left(2\frac{1\mp R^2}{1\pm R^2}-a\right)n(a_n+a_{-n})+n^2(a_n-a_{-n})\right]\sin(n\theta).
\end{align*}
\end{pro}
\begin{proof}
Note that if $h$ satisfies \eqref{h-expression}, we have
\begin{align*}
wh'(w)=&\sum_{n\geq 1}\left\{n(a_n-a_{-n})\cos(n\theta)+in(a_n+a_{-n})\sin(n\theta)\right\},\\
w^2h''(w)=&\sum_{n\geq 1}\left\{[n^2(a_n+a_{-n})-n(a_n-a_{-n})]\cos(n\theta)+i[n^2(a_n-a_{-n})-n(a_n+a_{-n})]\sin(n\theta)\right\}.
\end{align*}

Using the expression of $\partial_f G(R,0,0)$ in Proposition \ref{prop-linop-1-a} we find
\begin{align*}
\partial_f G_a(R,0,0)h(w)=&\sum_{n\geq 1}\left[\left(\mp 4\frac{R^2}{(1\pm R^2)^2}+n^2\right)(a_n+a_{-n})+\left(\frac{3\mp R^2}{1\pm R^2}-1-a\right)n(a_n-a_{-n})\right]\cos(n\theta)\\
&+i\left[\left(\frac{3\mp R^2}{1\pm R^2}-1-a\right)n(a_n+a_{-n})+n^2(a_n-a_{-n})\right]\sin(n\theta)\\
=&\sum_{n\geq 1}\left[\left(\mp 4\frac{R^2}{(1\pm R^2)^2}+n^2\right)(a_n+a_{-n})+\left(2\frac{1\mp R^2}{1\pm R^2}-a\right)n(a_n-a_{-n})\right]\cos(n\theta)\\
&+i\left[\left(2\frac{1\mp R^2}{1\pm R^2}-a\right)n(a_n+a_{-n})+n^2(a_n-a_{-n})\right]\sin(n\theta).
\end{align*}
\end{proof}

\subsection{Properties of the linearized operator}
We shall start by studying the kernel of $\partial_{(\lambda,f)}G_a(R,0,0)$, for $a\neq 0$. As for the case $a=0$, we need this preliminary lemma. Define
\begin{equation}\label{beta}
\beta(n,R):=\frac{n-2\frac{1\mp R^2}{1\pm R^2}+a}{n+2\frac{1\mp R^2}{1\pm R^2}-a}.
\end{equation}
\begin{lem}\label{lem-kernel-a}
The system 
\begin{equation}\label{kernel-system-a}
\left(
\begin{array}{cc}
\mp 4\frac{ R^2}{(1\pm R^2)^2}+n^2 & \left(2\frac{1\mp R^2}{1\pm R^2}-a\right)n\\
\left(2\frac{1\mp R^2}{1\pm R^2}-a\right)n & n^2
\end{array}
\right)
\left(
\begin{array}{l}
a_n+a_{-n}\\ a_n-a_{-n}
\end{array}
\right)
=
\left(
\begin{array}{l}
0\\ 0
\end{array}
\right), \quad n\geq 1,
\end{equation}
satisfying the following.
\begin{itemize}
\item (Euclidean case) For $N\geq 1\Bbb Z$ such that $R_N^{e,-}>0$, where
$$
R_N^{e,-}:=\sqrt{\frac{- \left(N^2+2-a^2\right)- 2\sqrt{3N^2-3+a^2}}{N^2-4-4a-a^2}},
$$
then the solution for $R=R_N^{e,-}$ is given by
$$
a_n=a_{-n}=0,\, n\neq N, \quad a_N=\beta(N, R_N^{e,-})a_{-N},\,\,  a_{-N}\in\R.
$$
Moreover, for $a<1$, $N=1$ and $R=R_1^{e,+}$, with
$$
R_1^{e,+}=\sqrt{\frac{-3+a^2+ 2a}{-3-4a-a^2}},
$$
we have a nontrivial solution given by
$$
a_n=a_{-n}=0,\, n\neq 1, \quad a_1=\beta(1,R^{e,+}_1)a_{-1},\,\,  a_{-1}\in\R.
$$
Otherwise, the only solution to \eqref{kernel-system-a} is the trivial one.
\item (Hyperbolic case) If $R=R_{N}^{h,+}$, with 
$$
R_n^{h,+}:=\sqrt{\frac{ \left(n^2+2-a^2\right)+ 2\sqrt{3n^2-3+a^2}}{n^2-4-4a-a^2}},
$$
and $N$ such that $R_{N}^{h,+}>0$, then we have a nontrivial solution given by
$$
a_n=a_{-n}=0,\, n\neq N, \quad a_N=\beta(N, R_N^{h,+})a_{-N},\,\,  a_{-N}\in\R.
$$
Moreover, if $R=R_{N}^{h,-}$ with
$$
R_n^{h,-}:=\sqrt{\frac{ \left(n^2+2-a^2\right)- 2\sqrt{3n^2-3+a^2}}{n^2-4-4a-a^2}},
$$ 
and for $N$ such that $R_{N}^{h,-}>0$, then we have a nontrivial trivial solution given by
$$
a_n=a_{-n}=0,\, n\neq N, \quad a_N=\beta(N, R_N^{h,-})a_{-N},\,\,  a_{-N}\in\R.
$$
Otherwise, the only solution to \eqref{kernel-system-a} is the trivial one.
\end{itemize}
\end{lem}
\begin{proof}
Denote now $b_n=a_n+a_{-n}$ and $c_n=a_n-a_{-n}$. Trivially, we have that 
$$a_n=\frac{b_n+c_n}{2},\quad a_{-n}=\frac{b_n-c_n}{2}.$$
In this way, \eqref{kernel-system-1-a} reads as
system for each mode $n\geq 1$:
\begin{equation}\label{kernel-system-2-a}
\left(
\begin{array}{cc}
\mp 4\frac{ R^2}{(1\pm R^2)^2}+n^2 & \left(2\frac{1\mp R^2}{1\pm R^2}-a\right)n\\
\left(2\frac{1\mp R^2}{1\pm R^2}-a\right)n & n^2
\end{array}
\right)
\left(
\begin{array}{l}
b_n\\ c_n
\end{array}
\right)
=
\left(
\begin{array}{l}
0\\ 0
\end{array}
\right).
\end{equation}
Hence, if $b_n=c_n=0$, then $a_n=a_{-n}$ and the kernel is trivial. In order to find nontrivial solutions (note that we need that the dimension is one in order to implement the Crandall-Rabinowitz theorem) we need to solve the equation
\begin{equation*}
\left|
\begin{array}{cc}
\mp 4\frac{ R^2}{(1\pm R^2)^2}+n^2 & \left(2\frac{1\mp R^2}{1\pm R^2}-a\right)n\\
\left(2\frac{1\mp R^2}{1\pm R^2}-a\right)n & n^2
\end{array}
\right|=0,
\end{equation*}
and that is
\begin{equation}\label{kernel-det-a}
n^2=\left(2\frac{1\mp R^2}{1\pm R^2}-a\right)^2\pm \frac{4 R^2}{(1\pm R^2)^2}.
\end{equation}
We can also solve \eqref{kernel-det-a} in terms of $n$. For the Euclidean case, we find two solutions given by
\begin{equation}\label{kernel-R-n-euclidean-a}
R^2=\frac{- \left(n^2+2-a^2\right)\pm 2\sqrt{3n^2-3+a^2}}{n^2-4-4a-a^2}.
\end{equation}
whereas for the Hyperbolic case we achieve
\begin{equation}\label{kernel-R-n-hyp-a}
R^2=\frac{ \left(n^2+2-a^2\right)\pm 2\sqrt{3n^2-3+a^2}}{n^2-4-4a-a^2}.
\end{equation}

Note that here we have the extra constrain $R^2>0$, which will give us a valid position only for some values of $n$, where $a$ is fixed.

\medskip\noindent $\bullet$ Euclidean case.

Recall that $a>0$ is fixed. Note that since $R^2>0$ we could have either that there is no solution for some $n$, we have only one solution or we could have two solutions.

Having $a>0$ fixed, there exists $R>0$ satisfying \eqref{kernel-R-n-euclidean-a} for $n$ satisfying the following.
\begin{itemize}
\item If $n<2-a$ (we could only have $n=1$ for small $a$ such that $a<1$), then we have the solution
$$
R_n^+=\sqrt{\frac{- \left(n^2+2-a^2\right)+ 2\sqrt{3n^2-3+a^2}}{n^2-4-4a-a^2}}.
$$
In this case $R_1^+$ is only possible:
$$
R_1^{e,+}=\sqrt{\frac{-3+a^2+ 2a}{-3-4a-a^2}}.
$$

\item We have the positive solution 
$$
R_n^{e,-}=\sqrt{\frac{- \left(n^2+2-a^2\right)- 2\sqrt{3n^2-3+a^2}}{n^2-4-4a-a^2}},
$$
in the cases:
\begin{itemize}
\item[i)] If $a^2-2<n^2$, then $n^2<(a+2)^2$
\item[ii)] If $a^2-2>n^2$, then $(a-2)^2<n^2<(a+2)^2$
\end{itemize}
\end{itemize}
Note that if $a<1$, then we have one solution given by $R_1^{e,+}$, and another one given by $R_n^{e,-}$ with $n<a+2$. 

In the case $a>1$, then there is one solution given by $R_n^{e,-}$ where $n$ has to satisfies $i)$ or $ii)$ above.

Hence, for fixed $a$, there exists a finite number of $n$ such that $R_n^{e,-}$ is a positive solution of \eqref{kernel-R-n-euclidean-a}. We can observe this fact by doing the limit as $n\rightarrow +\infty$ in \eqref{kernel-R-n-euclidean-a}, which will arrive to a negative $R$. Hence, here we say that the number of eigenvalues is finite.

Then, fixing $N$ such that $R_N^{e,-}>0$ and fixing $R=R_N^{e,-}$, we have that the system \eqref{kernel-system-2-a} has a non trivial solution. That is given by
$$
c_N=-b_N \frac{\left(2\frac{1-(R_N^{e,-})^2}{1+(R_N^{e,-})^2}-a\right)}{N},
$$
whereas $c_n=b_n=0$ for $n\neq N$. We can come back to the original variables $a_N$ and $a_{-N}$ by the following
\begin{align*}
2a_N=&\frac{b_N}{N}\left(N-2\frac{1-(R_N^{e,-})^2}{1+(R_N^{e,-})^2}+a\right),\\
2a_{-N}=&\frac{b_N}{N}\left(N+2\frac{1-(R_N^{e,-})^2}{1+(R_N^{e,-})^2}-a\right),
\end{align*}
and then
$$
a_N=\frac{N-2\frac{1-(R_N^{e,-})^2}{1+(R_N^{e,-})^2}+a}{N+2\frac{1-(R_N^{e,-})^2}{1+(R_N^{e,-})^2}-a}a_{-N}.
$$

Note that in a similar way, we get a nontrivial solution with $R=R_1^{e,+}$ and $a<1$.

\medskip\noindent $\bullet$ Hyperbolic case.

Here, doing the limit as $n\rightarrow +\infty$ arrives to some $R_{\infty}$ positive. Then, the intuition is that there is an infinite number of $n$ such that \eqref{kernel-R-n-hyp-a} is positive. Let us analyze the solutions of $R_n$.

\begin{itemize}
\item We have the solution
$$
R_n^{h,-}=\sqrt{\frac{ \left(n^2+2-a^2\right)- 2\sqrt{3n^2-3+a^2}}{n^2-4-4a-a^2}},
$$
for $n^2>(a-2)^2$.
\item We have the solution 
$$
R_n^{h,+}=\sqrt{\frac{ \left(n^2+2-a^2\right)+ 2\sqrt{3n^2-3+a^2}}{n^2-4-4a-a^2}},
$$
for $n<a-2$ or $n>a+2$.
\end{itemize}

Hence, for $R=R^{h,-}_N$, with $N^2>(a-2)^2$ we get non trivial solutions of the system. Same happen for $R=R^{h,+}_N$ with $N<a-2$ or $N>a+2$.

\end{proof}

In order to use the previous lemma, we need to know the values of $n$ such that $R_n^{e,-}$, $R_n^{h,-}$ and $R_n^{h,+}$ are real numbers and positive. This was studied in the previous proof and we write it in the following lemma for the sake of clarity.


Define the admissible set for $n$:
$$
\mathscr{A}(R^\star_n)=\left\{n\geq 1, \quad  R^\star_n>0\right\}\cap\N,
$$
which will be studied in the following lemma.

\begin{lem}\label{lem-m}
The following assertions hold true.
\begin{itemize} 
\item If $a>\sqrt{2}$ then $R_n^{e,-}>0$ for $\sqrt{a^2-2}<n<a+2$. If $a\leq\sqrt{2}$, then $R_n^{e,-}>0$ for $n<a+2$.
\item If $a\in(0,1)$ then $R_n^{h,-}>0$ for $n>2-a$. If $a\geq 1$ then $R_n^{h,-}>0$ for $n\geq 1$. 
\item $R_n^{h,+}>0$ for $n<a-2$ or $n>a+2$.
\end{itemize}
\end{lem}

Moreover, due to the expression of $R_n^{e,-}$, $R_n^{h,+}$ and $R_n^{h,-}$ we have that all of them are monotone with respect to $n$, which is needed in order to have a one-dimensional kernel for the linearized operator.

In the following proposition, we find the eigenvalues associated to our problem. We observe different situations depending on the euclidean or hiperbolic case.

\begin{pro}\label{prop-kernel-a}
The following assertions hold true.
\begin{itemize}
\item (Euclidean case) If $R=R_1^{e,+}$, then 
\begin{align*}
\textnormal{Ker}\left[\partial_{(\lambda,f)} G_a(R_1^{e,+},0,0)\right]=\left\{\kappa(\beta(1,R_1^{e,+})w+\overline{w}),\quad \kappa\in\R\right\}.
\end{align*}
If $R=R_N^{e,-}$, with $N\in\mathscr{A}(R_N^{e,-})$, then 
\begin{align*}
\textnormal{Ker}\left[\partial_{(\lambda,f)} G_a(R_N^{e,-},0,0)\right]=\left\{\kappa(\beta(N,R_N^{e,-})w^N+\overline{w}^N),\quad \kappa\in\R\right\}.
\end{align*}
Otherwise, the kernel is trivial.
\item (Hyperbolic case) If $R=R_N^{h,-}$, with $N\in\mathscr{A}(R_N^{h,-})$, then
\begin{align*}
\textnormal{Ker}\left[\partial_{(\lambda,f)} G_a(R_N^{h,-},0,0)\right]=\left\{\kappa(\beta(N,R_N^{h,-})w^N+\overline{w}^N),\quad \kappa\in\R\right\}.
\end{align*}
If $R=R_N^{h,+}$, with $N\in\mathscr{A}(R_N^{h,+})$, then
\begin{align*}
\textnormal{Ker}\left[\partial_{(\lambda,f)} G_a(R_N^{h,+},0,0)\right]=\left\{\kappa(\beta(N,R_N^{h,+})w^N+\overline{w}^N),\quad \kappa\in\R\right\}.
\end{align*}
Otherwise, the kernel is trivial.
\end{itemize}
\end{pro}
\begin{proof}
In order to check the kernel of $\textnormal{Ker}\left[\partial_{(\lambda,f)} G_a(R,0,0)\right]$, we should study the following equation 
$$
\partial_\lambda G_a(R,0,0)\mu+\partial_f G_a(R,0,0)h=0.
$$
We trivially find that $\mu=0$. Then, by virtue of Proposition \ref{prop-linop-2-a}, it amounts to study the following system for each mode $n\geq 1$:
\begin{equation}\label{kernel-system-1-a}
\left(
\begin{array}{cc}
\mp 4\frac{ R^2}{(1\pm R^2)^2}+n^2 & \left(2\frac{1\mp R^2}{1\pm R^2}-a\right)n\\
\left(2\frac{1\mp R^2}{1\pm R^2}-a\right)n & n^2
\end{array}
\right)
\left(
\begin{array}{l}
a_n+a_{-n}\\ a_n-a_{-n}
\end{array}
\right)
=
\left(
\begin{array}{l}
0\\ 0
\end{array}
\right).
\end{equation}
Note that such system was studied in Lemma \ref{lem-kernel-a}. With the assertions of Lemma \ref{lem-kernel-a} together with the monotonicity of the eigenvalues with respect to $n$ we conclude the proof.
\end{proof}

Next, we shall characterize the range of the linearized operator.

\begin{pro}\label{prop-range-a}
For $N$ and $R^\star$ such that \eqref{kernel-det-a} is satisfied, we get
$$
\textnormal{Range}\left[\partial_{\lambda,f}G_a(R^\star,0,0)\right]=\left\{f\in Y, \quad 
f_N=f_{-N}\frac{N+2\frac{1\mp (R^\star)^2}{1\pm (R^\star)^2}-a}{N-2\frac{1\mp (R^\star)^2}{1\pm (R^\star)^2}+a}
\right\}.
$$
Otherwise, the range is $Y$.
\end{pro}
\begin{proof}
Since the linear operator is Fredholm of zero index by Proposition \ref{prop-fredholm-a} we know that the codimension of the range is $1$. Moreover,  in order to study the range of $\partial_{(\lambda,f)} G_a(R,0,0)$, we should study the following equation 
$$
\partial_\lambda G_a(R,0,0)\mu+\partial_f G_a(R,0,0)h=d,
$$
for $d\in Y$. We can write $d$ in Fourier modes as
$$
d(w)=\sum_{n\in\Z}d_n w^n.
$$
Then, we trivially find that $\mu=d_0$. By virtue of Proposition \ref{prop-linop-2-a}, we have to study the following system
\begin{equation}\label{range-system-1-a}
\left(
\begin{array}{cc}
\mp 4\frac{ R^2}{(1\pm R^2)^2}+n^2 & \left(2\frac{1\mp R^2}{1\pm R^2}-a\right)n\\
\left(2\frac{1\mp R^2}{1\pm R^2}-a\right)n & n^2
\end{array}
\right)
\left(
\begin{array}{l}
a_n+a_{-n}\\ a_n-a_{-n}
\end{array}
\right)
=
\left(
\begin{array}{l}
d_n+d_{-n}\\ d_n-d_{-n}
\end{array}
\right).
\end{equation}
Note that such matrix has been studied in Lemma \ref{lem-kernel-a}.

Denote $b_n=a_n+a_{-n}$, $c_n=a_n-a_{-n}$, $e_n=d_n+d_{-n}$ and $f_n=d_n-d_{-n}$. Take $N$ and $R=R^\star$ such that \eqref{kernel-det-a} is satisfied. Then, the kernel equation (that is, the homogeneous system) has nontrivial solutions and are given by
$$
c_N=-\frac{b_N}{N}\left(2\frac{1\mp (R^\star)^2}{1\pm (R^\star)^2}-a\right),
$$
and $c_n=b_n=0$ for $n\neq N$. Hence, the range equation \eqref{range-system-1-a} (i.e., the complete system) has a solution if
$$
e_N=\frac{f_N}{N}\left(2\frac{1\mp (R^\star)^2}{1\pm (R^\star)^2}-a\right),
$$
and has always a solution for any $e_n$ and $f_n$, with $n\neq N$. Coming back to $d_N$ and $d_{-N}$ we find the condition
$$
d_N=d_{-N}\frac{N+2\frac{1\mp (R^\star)^2}{1\pm (R^\star)^2}-a}{N-2\frac{1\mp (R^\star)^2}{1\pm (R^\star)^2}+a}.
$$
\end{proof}

\subsection{Transversal condition}
Finally, we check the transversal condition of the Crandall-Rabinowitz theorem.
\begin{pro}If
\begin{align}\label{condition-transversal}
\left(\frac{(N+a)(1\pm (R^\star)^2)-2(1\mp (R^\star)^2)}{(N-a)(1\pm (R^\star)^2)+2(1\mp (R^\star)^2)}\right)^2
&\neq \frac{ 1\mp (R^\star)^3-2N(1\pm (R^\star)^2) }{ 1\mp (R^\star)^3+2N(1\pm (R^\star)^2)   },
\end{align}
then the transversal condition is satisfied, that is,
$$
\partial_R \partial_{(\lambda,f)} G_a(R^*,0,0)(\lambda,h)^*\notin \textnormal{Range}\left[\partial_{(\lambda,f)}G_a(R^*,0,0)\right],
$$
where $R^*$ is fixed such that \eqref{kernel-det-a} is satisfied for $N$ fixed. Here, $h^*$ is an element of $\textnormal{Ker}\left[\partial_{(\lambda,f)}G_a(R^*,0,0)\right]$.
\end{pro}

\begin{proof}
Here, we can use the computations done for $a=0$ because $\partial_R\partial_f G_a(R^\star,0,0)=\partial_R\partial_f G_0(R^\star,0,0)$. Then, we find:

From Proposition \ref{prop-linop-2-a} we have
\begin{align*}
\partial_R \partial_{(\lambda,f)}G_a(R^\star,0,0)h^\star(w)=&\left[\mp 8R^\star\frac{1\mp (R^\star)^3}{(1\pm (R^\star)^2)^3} (1+\beta(N, R^\star))\mp 8 \frac{R^\star}{(1\pm (R^\star)^2)^2} N(\beta(N,R^\star)-1)\right]\cos(N\theta)\\
&+i\left[\mp 8 \frac{R^\star}{(1\pm (R^\star)^2)^2}N(1+\beta(N,R^\star))\right]\sin(N\theta)\\
=: e_N\cos(N\theta)+if_N\sin(N\theta).
\end{align*}
Define $g_N$ and $g_{-N}$ as
$$
g_N+g_{-N}=e_N, \quad g_N-g_{-N}=f_N,
$$
which agrees with
$$
g_N=\frac{e_N+f_N}{2},\quad g_{-N}=\frac{e_N-f_N}{2}.
$$
Then, by Proposition \ref{prop-range-a}, we have that $c_N\cos(N\theta)+i d_N\sin(N\theta)\notin \textnormal{Range}\left[\partial_{(\lambda,f)}G_a(R^*,0,0)\right]$ if 
$$
g_N\neq g_{-N}\frac{N+2\frac{1\mp (R^\star)^2}{1\pm (R^\star)^2}-a}{N-2\frac{1\mp (R^\star)^2}{1\pm (R^\star)^2}+a}
$$
Using the definition of $g_N$ and $g_{-N}$ we have
$$
e_N+f_N\neq (e_N-f_N)\frac{N+2\frac{1\mp (R^\star)^2}{1\pm (R^\star)^2}-a}{N-2\frac{1\mp (R^\star)^2}{1\pm (R^\star)^2}+a}.
$$
Now, we have to use the expression of $c_N$, $d_N$ and $R^\star$.
\begin{align}\label{cond-trans-a}
&\left[\mp 8R^\star\frac{1\mp (R^\star)^3}{(1\pm (R^\star)^2)^3} (1+\beta(N, R^\star))\mp 16 \frac{R^\star}{(1\pm (R^\star)^2)^2} N \beta(N,R^\star)\right]\nonumber\\
&\neq \left[\mp 8R^\star\frac{1\mp (R^\star)^3}{(1\pm (R^\star)^2)^3} (1+\beta(N, R^\star))\pm 16 \frac{R^\star}{(1\pm (R^\star)^2)^2} N\right]\frac{N+2\frac{1\mp (R^\star)^2}{1\pm (R^\star)^2}-a}{N-2\frac{1\mp (R^\star)^2}{1\pm (R^\star)^2}+a},
\end{align}
where we recall
$$
\beta(N, R^\star)=\frac{N-2\frac{1\mp (R^\star)^2}{1\pm (R^\star)^2}+a}{N+2\frac{1\mp (R^\star)^2}{1\pm (R^\star)^2}-a}.
$$
Then, the condition \eqref{cond-trans-a} follows as
\begin{align*}
\beta(N, R^\star)^2
&\neq \frac{ R^\star\frac{1\mp (R^\star)^3}{(1\pm (R^\star)^2)^3} - 2 \frac{R^\star}{(1\pm (R^\star)^2)^2} N}{ R^\star\frac{1\mp (R^\star)^3}{(1\pm (R^\star)^2)^3} + 2 \frac{R^\star}{(1\pm (R^\star)^2)^2} N },
\end{align*}
agreeing with
\begin{align*}
\beta(N, R^\star)^2
&\neq \frac{ 1\mp (R^\star)^3-2N(1\pm (R^\star)^2) }{ 1\mp (R^\star)^3+2N(1\pm (R^\star)^2)   }.
\end{align*}
Using now the definition of $\beta$ we find
\begin{align*}
\left(\frac{(N+a)(1\pm (R^\star)^2)-2(1\mp (R^\star)^2)}{(N-a)(1\pm (R^\star)^2)+2(1\mp (R^\star)^2)}\right)^2
&\neq \frac{ 1\mp (R^\star)^3-2N(1\pm (R^\star)^2) }{ 1\mp (R^\star)^3+2N(1\pm (R^\star)^2)   }.
\end{align*}
\end{proof}
Let us mention that \eqref{condition-transversal} was checked for the case $a=0$ in Proposition \ref{prop-transversal}, and can be also checked for other $a\neq 0$.

\section{Main result: $m$-fold symmetry solutions}\label{sec-main-result}
This section aims to give the main results of this work. First, let us introduce the $m$-fold symmetry in the function spaces defined in \eqref{X} and \eqref{Y}. For any $m\geq 1$, define
\begin{align*}
X_m:=&\left\{f\in C^2(\T), \quad f(w)=\sum_{n\in Z, n\neq 0}f_n w^{nm}, f_n\in\R\right\},\\
Y_m:=&\left\{f\in C^0(\T), \quad f(w)=\sum_{n\in Z}f_n w^{nm}, f_{n}\in\R\right\}.
\end{align*}
Note that $f\in X_m$ verifies
$$
f(e^{i\frac{2\pi}{m}}w)=f(w),
$$
which means that $f$ is $m$-fold symmetry. If we come back to the original function $z$ in \eqref{z-equation-3} via \eqref{f} we have that $z$ verifies
$$
z(s+\frac{2\pi}{m})=e^{\frac{i2\pi}{m}}z(s),
$$
and then, the curve is invariant under rotations of angle $\frac{2\pi}{m}$. With the definition of such spaces, we have that the nonlinear operator $G_a$ is also well-defined.
\begin{pro}\label{prop-well-defined-2}
For any $m\geq 1$, the functional $G_a:\mathcal{R}\times \R\times X_m\rightarrow Y_m$ is well-defined and $C^1$.
\end{pro}

Then, we are ready to announce our main result.

\begin{theo}[$a=0$]\label{theorem-main}
The following assertions hold true.
\begin{itemize}
\item (Euclidean case) For $m=1$, there exists $\delta>0$ such that for any $R\in (1-\delta, 1+\delta)$ there is nontrivial $(\Omega_1(R),f_1(R))\in\R\times X_1$, such that $G_0(R, \Omega_1(R),f_1(R))=0$, which bifurcates from $R=1$ (the circle).
\item (Hyperbolic case) For any $m\geq 3$, there exists $\delta>0$ such that for any $R\in (R_m^h-\delta, R_m^h+\delta)$ there is nontrivial $(\Omega_m(R),f_m(R))\in\R\times X_m$, such that $G_0(R, \Omega_m(R),f_m(R))=0$, which bifurcates from $R=R_m^h$ defined in \eqref{kernel-R-n-hyp-3-prop}.
\end{itemize}
\end{theo}
\begin{proof}
In the euclidean case, the proof follows from the Crandall-Rabinowitz theorem using the spaces $X_1$: all the hypothesis of the theorem were proved in Section \ref{sec-spectral-0}.

Let us give the idea of the proof for the hyperbolic case. Fix $m>3$ and $R=R_m^h$. By Proposition \ref{prop-well-defined-2} we have that $G:\mathcal{R}\times\mathbb{R}\times X_m\rightarrow Y_m$ is well-defined and $C^1$. Then, we move to prove the spectral properties of the Crandall-Rabinowitz theorem. In order to use there all the previous works, we have to slightly modify Section \ref{sec-spectral-0}. Note that there, we have taken $h$ to be written in Fourier series as \eqref{h-expression}:
\begin{equation*}
h(w)=\sum_{n\in\Z, n\neq 0}a_n w^n=\sum_{n\geq 1}\left\{a_n w^n+a_{-n}\overline{w}^n\right\}.
\end{equation*}
However, by introducing $m$ in the space, we have to take
\begin{equation*}
h(w)=\sum_{n\in\Z, n\neq 0}a_n w^{nm}=\sum_{n\geq 1}\left\{a_{n} w^{nm}+a_{-n}\overline{w}^{nm}\right\}.
\end{equation*}
Then we just have to change $n$ by $nm$ everywhere.

By Proposition \ref{prop-kernel} we can study the kernel of $\partial_{(\lambda,f)}G_0(R^h_m,0,0)$. The trivial solution for the kernel equation will come when $n=1$ and then:
\begin{align*}
\textnormal{Ker}\left[\partial_{(\lambda,f)} G_0(R_m^h,0,0)\right]=\left\{\kappa(\alpha_m w^m+\overline{w}^m),\quad \kappa\in\R\right\}.
\end{align*}
We can work in the same way with the range, and thus Proposition \ref{prop-range} implies
\begin{align*}
\textnormal{Range}\left[\partial_{(\lambda,f)} G_0(R_m^h,0,0)\right]=\left\{f\in Y_m, \quad f_{-1}=\frac{\sqrt{1+(R_m^h)^4+(R_m^h)^2}}{1-(R_m^h)^2}f_1\right\}.
\end{align*}
Then, both dimension of the kernel and the codimension of the range equals to 1, which is required to apply the Crandall-Rabinowitz theorem. Finally, Proposition \ref{prop-transversal} gives us that the transversal condition is satisfied.

Then, the Crandall-Rabinowitz theorem states that there exists a one dimensional curve, which is transversal to the trivial one (the helix), of roots of $G_0$.
\end{proof}

\begin{rem}
\begin{itemize}
\item In the hyperbolic case, we have found a nontrivial curve of solutions that bifurcates from $R=R_m^h$ given by \eqref{kernel-R-n-hyp-3-prop}
\begin{equation*}
R_m^h:=\sqrt{\frac{ \left(m^2+2\right)- 2\sqrt{3}\sqrt{m^2-1}}{m^2-4}},\quad N\geq 3.
\end{equation*}
Note that such eigenvalues increases to $1$, which is the limit point in the hyperbolic case, as $m$ increases.

\item We have obtained a nontrivial root $f$ of $G_0$, which coming back to the original variable $z$ means
$$
z_0(s)=g(e^{is})=g(w)=w(R+f(w)).
$$
\end{itemize}
\end{rem}

Finally, let us state the main theorem for all values of $a$, whose proof follows the same ideas using the spectral study done in Section \ref{sec-spectral-a}.

\begin{theo}[$a\neq 0$]\label{theorem-main-a}
Assume that \eqref{transversal} is satisfied for any case. For any $a>0$, the following assertions hold true.
\begin{itemize}
\item (Euclidean case) For $a\in(0,1)$ and $m=1$, there exists $\delta>0$ such that for any $R\in (R_1^{e,+}-\delta, R_1^{e,+}+\delta)$ there is a nontrivial $(\Omega_1(R),f_1(R))\in\R\times X_1$, such that $G_a(R, \Omega_1(R),f_1(R))=0$, which bifurcates from $R=R_1^{e,+}$.
\item (Euclidean case) For $m\geq 1$ such that $R_m^{e,-}>0$, there exists $\delta>0$ such that for any $R\in (R_m^{e,-}-\delta, R_m^{e,-}+\delta)$ there is a nontrivial $(\Omega_m(R),f_m(R))\in\R\times X_m$, such that $G_a(R, \Omega_m(R),f_m(R))=0$, which bifurcates from $R=R_m^{e,-}$.
\item (Hyperbolic case) For $m\geq 1$ such that $R_m^{h,-}>0$, there exists $\delta>0$ such that for any $R\in (R_m^{h,-}-\delta, R_m^{h,-}+\delta)$ there is a nontrivial $(\Omega_m(R),f_m(R))\in\R\times X_m$, such that $G_a(R, \Omega_m(R),f_m(R))=0$, which bifurcates from $R=R_m^{h,-}$.
\item (Hyperbolic case) For $m\geq 1$ such that $R_m^{h,+}>0$, there exists $\delta>0$ such that for any $R\in (R_m^{h,+}-\delta, R_m^{h,+}+\delta)$ there is a nontrivial $(\Omega_m(R),f_m(R))\in\R\times X_m$, such that $G_a(R, \Omega_m(R),f_m(R))=0$, which bifurcates from $R=R_m^{h,+}$.
\end{itemize}
\end{theo}

\begin{rem}

\begin{itemize}
\item Note that not all symmetries are allowed here. Indeed, for each case in the previous theorem we can bifurcate for those $m$ such that $R^\star>0$ (where here $R^\star_n= R_1^{e,+}, R_m^{e,-}, R_m^{h,-}, R_,^{h,+}$). Those values of $m$ are studied in Lemma \ref{lem-m} ang belongs to the admissible set $\mathscr{A}(R^\star_n)$.

By virtue of Lemma \ref{lem-m}, we have that there exists a finite number of symmetries $m$ for the euclidean case. However, we find an infinite number of curves parametrized by $m$ in the hyperbolic case.

\item As in the previous theorem for $a=0$, we can recover the original solution $z$ in terms of $f$:
$$
z_0(s)=g(e^{is})=g(w)=w(R+f(w)).
$$
\end{itemize}

\end{rem}

\subsection{Relation with Kida's solutions}\label{sec-kida}

In \cite{Kida81, Kida82}, the author looked for steady solutions of the vortex filament equation (only in the euclidean case):
\begin{equation}\label{Kida-2}
{\bf X}(t,s)=\mathcal{R}^{\Omega t}{\bf X^0}(s-at)+V(0,0,t),
\end{equation}
for constants $a, V$ and $\Omega$ referring respectively to the slipping motion, a constant speed translation and a rotation around the vertical axis. Moreover, he uses cilindrical  coordinates and writes
$$
{\bf X^0}(s)=(r(s)e^{i\theta(s)}, z(s)).
$$
Hence, we can compute the tangent vector as
\begin{equation}\label{Kida-3}
{\bf T^0}(s)=((r'(s)+i\theta'(s) r(s))e^{i\theta(s)},z'(s)),
\end{equation}
satisfying
\begin{equation}\label{Kida-1}
{\bf T}(t,s)=\mathcal{R}^{\Omega t}{\bf T^0}(s-at).
\end{equation}
As a consequence and using the so-called Hasimoto transformation \cite{Ha} one can easily obtained travelling wave solutions for the 1d cubic non-linear Schrödinger equation \cite{Scott}.
In the works \cite{Kida81, Kida82, Jerrard-Smets}, the authors work with the equation for ${\bf X}$, while we do it with the Schr\"odinger map equation. Hence we should compare the initial tangent vector. Moreover, note that even though \eqref{Kida-1}  is satisfied also for our solutions,\eqref{Kida-2} does not hold in general in our construction. Let us therefore justify that our main result provides some different steady solutions for the Schr\"odinger map equation to the ones obtained by Kida.

In \cite{Kida81,Kida82,Jerrard-Smets}, they found steady solutions of the type \eqref{Kida-1} with initial data \eqref{Kida-3} satisfying
\begin{align}
z'(s)=&\frac{\Omega}{2}[A-r^2(s)],\label{Kida-z}\\
\theta'(s)=&\frac{1}{2}V+(a-\frac12 AV\Omega)/(\Omega r^2(s)),\label{Kida-theta}\\
(r'(s))^2=&-\frac14\frac{g(r^2)}{r^2},\label{Kida-r}
\end{align}
for some constant $A$, and $g(R)$ is a polynomial of degree three, with two positive roots and one negative one, whose expression is given in \cite{Kida81} as
\begin{equation}\label{g-poly}
g(R)=R^3+(V^2-2A)R^2+(A^2-4-2AV^2+4Va)R+(2a-AV)^2.
\end{equation}
Denote now
$$
{\bf T^0}=(T_h, T_3),
$$
where $T_h=(T_1,T_2)$, that is, the horizontal variables. Hence, using \eqref{Kida-3} together with \eqref{Kida-z}--\eqref{Kida-r} we find 
$$
r^2(s)=A-\frac{2}{\Omega}T_3,
$$
amounting to the following relation between $T_h$ and $T_3$:
\begin{equation}\label{Kida-5}
|T_h|^2=(A-\frac{2}{\Omega}T_3)\left[\frac{V}{2}+\frac{a-\frac12 AV\Omega}{\Omega(A-\frac{2}{\Omega}T_3)}\right]^2-\frac14 \frac{g(A-\frac{2}{\Omega}T_3)}{A-\frac{2}{\Omega}T_3},
\end{equation}
for any $s$.

Let us check if our solutions satisfy \eqref{Kida-5}, which is indeed related to the perturbation chosen by Kida in \cite{Kida81,Kida82}. Note that our ansatz for the initial tangent vector is different and it can be written as \eqref{T0}:
\begin{equation}\label{T0-2}
{\bf T^0}(s)=\frac{1}{1+|z_0(s)|^2}(2z_0(s),1-|z_0(s)|^2),
\end{equation}
with
$$
z_0(s)=e^{is}(R+f(e^{is})),
$$
and $f$ being the perturbation. Let us check if \eqref{Kida-5} is satisfied in \eqref{T0-2}. Denote $\gamma:=|z_0(s)|^2$ and $\beta:=A-\frac{2}{\Omega}\frac{1-\gamma}{1+\gamma}$ to simplify the notation, hence \eqref{Kida-5} (which is satisfied by the Kida's solutions) is verified for our solutions if
$$
\frac{4\gamma}{(1+\gamma)^2}=\beta\left[\frac{V}{2}+\frac{a-\frac12 AV\Omega}{\Omega}\frac{1}{\beta}\right]^2-\frac14\frac{g(\beta)}{\beta}.
$$
From the definition of $\beta$ one has
$$
\gamma=\frac{A-\frac{2}{\Omega}-\beta}{\beta-\frac{2}{\Omega}},\quad 1+\gamma=\frac{A-\frac{4}{\Omega}}{\beta-\frac{2}{\Omega}},
$$
and then the condition amounts to
\begin{equation}\label{Kida-6}
4\beta\frac{(A-\frac{2}{\Omega}-\beta)(\beta-\frac{2}{\Omega})}{A-\frac{4}{\Omega}}=\frac{V^2}{4}\beta^2+\frac{(a-\frac12 AV\Omega)^2}{\Omega^2}+\beta V\frac{a-\frac12 AV\Omega}{\Omega}-\frac14g(\beta).
\end{equation}
Now let us insert the expression of $g$ in \eqref{g-poly} in \eqref{Kida-6}, finding that \eqref{Kida-6} agrees with
\begin{align}
\nonumber&\beta^3\left\{-\frac14+\frac{4}{A-\frac{4}{\Omega}}\right\}+\beta^2\left\{\frac{A}{2}-\frac{4}{A-\frac{4}{\Omega}}\right\}\\
&+\beta\left\{V\frac{a-\frac12 AV\Omega}{\Omega}-\frac14(A^2-4-2AV^2+4Va)+\frac{8}{A-\frac{4}{\Omega}}\frac{1}{\Omega}(A-\frac{2}{\Omega})\right\}\nonumber\\
&+\frac{(a-\frac12 AV\Omega)^2}{\Omega^2}-\frac14 (2a-AV)^2=0.\label{Kida-pol}
\end{align}
Assume that the previous polynomial is not trivial, that is, at least one of the coefficients is non vanishing. In that case, $\beta=\beta(s)$ is a root of a polynomial of degree three and hence is constant. Then, using the definition of $\beta$ we achieve that $\gamma$ is also constant in $s$ and then $|z_0(s)|$ is constant in $s$. However, since $z_0$ is a perturbation of the trivial solution one gets
\begin{equation}\label{Kida-7}
|z_0(s)|=|R+f(e^{is})|,
\end{equation}
which is constant only if $f=0$, since $s\mapsto e^{is}(R+f(e^{is}))$ is a parametrization of a $2D$ curve whose modulus is constant and thus it is radial.

Notice that we have assumed before that the polynomial \eqref{Kida-pol} is nontrivial. Assuming that it is trivial, we arrive to a system of four equations with four unknowns. In order to have that the coefficient with $\beta^3$ is vanishing we get
$$
A=\frac{4(1+4\Omega)}{\Omega},
$$
and in order to have that the term without $\beta$ is vanishing one gets $\Omega=1$ implying $A=20$. Now, assuming that the coefficient with $\beta^2$ we get should have
$$
\frac{A}{2}=\frac{4}{A-\frac{4}{\Omega}},
$$
which is not compatible with $\Omega=1$ and $A=20$. Hence, we conclude that there is no choice of $(a,A,\Omega,V)$ such that the polynomial \eqref{Kida-pol} is trivial.

\end{document}